\newif\ifmarek
\marektrue

\documentclass[reqno,twoside,11pt]{amsart}

\usepackage{srcltx}
\usepackage{amsmath}
\usepackage{amsfonts}
\usepackage{amssymb}
\usepackage{verbatim}
\usepackage{epsfig}
\usepackage{color}

\IfFileExists{epsf.def}{\input epsf.def}{\usepackage{epsf}}

\DeclareFontFamily{OT1}{eusb}{} \DeclareFontShape{OT1}{eusb}{m}{n} {<5> <6> <7> <8> <9> <10> <11> <12> <14.4> eusb10}{}
\DeclareMathAlphabet{\eusb}{OT1}{eusb}{m}{n}

\DeclareFontFamily{OT1}{eusm}{} \DeclareFontShape{OT1}{eusm}{m}{n} {<5> <6> <7> <8> <9> <10> <11> <12> <14.4> eusm10}{}
\DeclareMathAlphabet{\eusm}{OT1}{eusm}{m}{n}

\DeclareFontFamily{OT1}{eufm}{} \DeclareFontShape{OT1}{eufm}{m}{n} {<5> <6> <7> <8> <9> <10> <11> <12> <14.4> eufm10}{}
\DeclareMathAlphabet{\mathfrak}{OT1}{eufm}{m}{n}

\DeclareFontFamily{OT1}{fraktura}{}
\DeclareFontShape{OT1}{fraktura}{m}{n} {<5> <6> <7> <8> <9> <10> <11> <12> <13> <14.4> [1.1] eufm10}{}
\DeclareMathAlphabet{\fraktura}{OT1}{fraktura}{m}{n}

\DeclareFontFamily{OT1}{cmfi}{} \DeclareFontShape{OT1}{cmfi}{m}{n} {<5> <6> <7> <8> <9> <10> <11> <12> <13> <14.4> [0.9] cmfi10}{}
\DeclareMathAlphabet{\cmfi}{OT1}{cmfi}{b}{n}

\DeclareFontFamily{OT1}{cmss}{} \DeclareFontShape{OT1}{cmss}{m}{n} {<5> <6> <7> <8> <9> <10> <11> <12> <13> <14.4> cmss10}{}
\DeclareMathAlphabet{\cmss}{OT1}{cmss}{m}{n}

\IfFileExists{mathptmx.sty}{\usepackage{mathptmx}\usepackage{mathrsfs}}
{\usepackage{times}\usepackage{mathrsfs}}
\renewcommand{\mathcal}{\eusm}

\newtheoremstyle{thm}{1.5ex}{1.5ex}{\itshape\rmfamily}{} {\bfseries\rmfamily}{}{2ex}{}

\newtheoremstyle{def}{1.5ex}{1.5ex}{\slshape\rmfamily}{} {\bfseries\rmfamily}{}{2ex}{}

\newtheoremstyle{rem}{1.5ex}{1.5ex}{\rmfamily}{} {\bfseries\rmfamily}{} {1.5ex}{}

\newenvironment{proofsect}[1] {\vskip0.1cm\noindent{\rmfamily\itshape#1.}}{\qed\vspace{0.15cm}}

\theoremstyle{thm}
\newtheorem{theorem}{Theorem}[section]
\newtheorem{lemma}[theorem]{Lemma}
\newtheorem{proposition}[theorem]{Proposition}

\newtheorem*{Main Theorem}{Main Theorem.}
\newtheorem{corollary}[theorem]{Corollary}
\newtheorem*{special theorem}{Lindeberg-Feller Theorem for Martingales}

\newtheorem{assumption}[theorem]{Assumption}

\theoremstyle{def}

\theoremstyle{rem}
\newtheorem{remark}[theorem]{{\bfseries Remark}}

\numberwithin{equation}{section}

\renewcommand{\section}{\secdef\sct\sect}
\newcommand{\sct}[2][default]{%
\refstepcounter{section}
\addcontentsline{toc}{section}{{\tocsection {}{\thesection}{\!\!\!\!#1\dotfill}}{}}
\vspace{0.7cm}
\centerline{\scshape\thesection.\ #1} \nopagebreak \vspace{0.2cm}}
\newcommand{\sect}[1]{%
\vspace{0.4cm} \centerline{\large\scshape\rmfamily #1}
\vspace{0.2cm}}

\renewcommand{\subsection}{\secdef\subsct\sbsect}
\newcommand{\subsct}[2][default]{\refstepcounter{subsection}
\addcontentsline{toc}{subsection}
{{\tocsection{\!\!}{\hspace{1.2em}\thesubsection}{\!\!\!\!#1\dotfill}}{}}
\nopagebreak\vspace{0.75\baselineskip} {\flushleft\bf
\thesubsection~\bf #1.~}
\\*[3mm]\noindent
\nopagebreak}
\newcommand{\sbsect}[1]{\nopagebreak\vspace{0.75\baselineskip}\noindent
\textbf{#1.~}\\*[3mm]\noindent
\nopagebreak}

\renewcommand{\subsubsection}{%
\secdef \subsubsect\sbsbsect}
\newcommand{\subsubsect}[2][default]{%
\refstepcounter{subsubsection} 
\addcontentsline{toc}{subsubsection}{{\tocsection{\!\!}
{\hspace{3.05em}\thesubsubsection}{\!\!\!\!#1\dotfill}}{}}
\nopagebreak
\vspace{0.15\baselineskip} \nopagebreak {\flushleft\rmfamily
\itshape\thesubsubsection
\ \rmfamily #1\/.}\ }
\newcommand{\sbsbsect}[1]{\vspace{0.1cm}\noindent
\rmfamily \itshape
\arabic{section}.\arabic{subsection}.\arabic{subsubsection} \
\sffamily #1\/.\ }

\renewcommand{\caption}[1]{%
\vglue0.5cm
\refstepcounter{figure}
\begin{minipage}{0.9\textwidth}\small {\sc Figure~\thefigure. }#1\end{minipage}}

\newcommand{\1}{\operatorname{\sf 1}\!}
\newcommand{\dist}{\operatorname{dist}}
\newcommand{\R}{\mathbb R}
\newcommand{\Z}{\mathbb Z}
\newcommand{\T}{\mathbb T}
\newcommand{\BbbP}{\mathbb P}
\newcommand{\E}{\mathbb E}
\newcommand{\FF}{\mathcal F}
\newcommand{\N}{\mathbb N}
\renewcommand{\d}{\operatorname{d}}
\newcommand{\cc}{{\text{\rm c}}}
\newcommand{\hate}{\hat{\text{\rm e}}}
\newcommand{\textd}{\text{\rm d}\mkern0.5mu}
\newcommand{\texte}{\text{\rm e}}
\newcommand{\texti}{\text{\rm  i}\mkern0.7mu}
\newcommand{\Deltad}{{\Delta^{\ssup{\textd}}}}
\newcommand{\DG}{\nabla^{(\textd)}}

\newcommand{\ssup}[1]{{\scriptscriptstyle{({#1}})}}
\newcommand{\wt}{\widetilde}
\newcommand{\supp}{\operatorname{supp}}
\newcommand{\twoeqref}[2]{(\ref{#1}--\ref{#2})}

\def\myffrac#1#2 in #3{\raise 2.6pt\hbox{$#3 #1$}\mkern-1.5mu\raise 0.8pt\hbox{$#3/$}\mkern-1.1mu\lower 1.5pt\hbox{$#3 #2$}}
\newcommand{\ffrac}[2]{\mathchoice%
{\myffrac{#1}{#2} in \scriptstyle}
{\myffrac{#1}{#2} in \scriptstyle}
{\myffrac{#1}{#2} in \scriptscriptstyle}
{\myffrac{#1}{#2} in \scriptscriptstyle}
}

\DeclareFontFamily{OT1}{cmss}{} \DeclareFontShape{OT1}{cmss}{m}{n} {<5> <6> <7> <8> <9> <10> <11> <12> <13> <14.4> cmss10}{}
\DeclareMathAlphabet{\cmss}{OT1}{cmss}{m}{n}

\begin{document}
\title[Eigenvalue fluctuations\hfill\qquad]{Eigenvalue fluctuations for lattice Anderson Hamiltonians: Unbounded potentials}
\author[\qquad \hfill Biskup, Fukushima, K\"onig]{Marek Biskup$^{1,2}$,\, Ryoki Fukushima$^{3}$\and\,\,Wolfgang K\"onig$^{4,5}$}

\thanks{\hglue-4.5mm\fontsize{9.6}{9.6}\selectfont\copyright\,2017 by M.~Biskup, R.~Fukushima and W.~K\"onig. Reproduction, by any means, of the entire article for non-commercial purposes is permitted without charge.
\vspace{2mm}
}
\thanks{\noindent This research has been partially supported by NSF grant
DMS-1712632, GA\v CR project P201/16-15238S, DFG Forschergruppe 718
``Analysis and Stochastics in Complex Physical Systems,'' JSPS KAKENHI Grant Number 24740055\&16K05200 and JSPS and DFG under the Japan-Germany Research Cooperative Program.}

\maketitle

\vglue-2mm

\centerline{\textit{$^1$Department of Mathematics, UCLA, Los Angeles, California, USA}}
\centerline{\textit{$^2$Center for Theoretical Study, Charles University, Prague, Czech Republic}}
\centerline{\textit{$^3$Research Institute in Mathematical Sciences, Kyoto University, Kyoto, Japan}}
\centerline{\textit{$^4$Weierstra\ss-Institut f\"ur Angewandte Analysis und Stochastik, Berlin, Germany}}
\centerline{\textit{$^5$Institut f\"ur Mathematik, Technische Universit\"at Berlin, Berlin, Germany}}

\vglue3mm

\begin{abstract}
We consider random Schr\"odinger operators with  Dirichlet
boundary conditions outside lattice approximations of a smooth Euclidean domain and study the behavior of its lowest-lying eigenvalues in the limit when the 
lattice spacing tends to zero. 
Under a suitable moment assumption on the random potential and regularity of the spatial dependence of its mean, we prove that the eigenvalues of the random operator converge to those of a deterministic 
Schr\"odinger operator. Assuming also regularity of the variance, the fluctuation of the random eigenvalues around their
mean are shown to obey a multivariate  central limit theorem.
This extends the authors' recent work 
where similar conclusions have been obtained for bounded random potentials. 
\end{abstract}

\section{Introduction and results}
\noindent
This note is a continuation of our recent paper~\cite{BFK14} where 
we studied the statistics of low-lying eigenvalues of 
Anderson Hamiltonians in the ``homogenization'' regime, i.e., under 
the conditions when a non-trivial continuum limit can be taken. 
The derivations of \cite{BFK14} were restricted to the class of bounded potentials; here we extend the main conclusions --- namely, the convergence of the individual eigenvalues to their continuum (and deterministic) counterparts as well as a proof of Gaussian fluctuations around their mean --- to a class of 
unbounded random potentials satisfying suitable, and essentially sharp, moment conditions.

Our setting is as follows: Let $D$ be a bounded open subset of~$\R^d$ 
whose boundary is $C^{1,\alpha}$ for some~$\alpha>0$. For any 
$\epsilon>0$, we define the discretized version of~$D$ as
\begin{equation}
\label{E:1.1}
D_\epsilon:=\bigl\{x\in\Z^d\colon \dist_\infty(\epsilon x,D^\cc)>\epsilon\bigr\},
\end{equation}
where $\dist_\infty$ is the $\ell^\infty$-distance in~$\R^d$.
Given any potential $\xi\colon D_\epsilon\to\R$, we now consider 
the linear operator (a matrix)~$H_{D_\epsilon,\xi}$ acting on the linear space 
of functions $f\colon\Z^d\to\R$ that vanish outside~$D_\epsilon$ 
via
\begin{equation}
\label{E:1.2}
(H_{D_\epsilon,\xi}f)(x):=-\epsilon^{-2}(\Deltad f)(x)
+\xi(x)f(x),\qquad x\in\Z^d,
\end{equation}
where $\Deltad$ is the lattice Laplacian
\begin{equation}
(\Deltad f)(x):=\sum_{y\colon |x-y|=1}\bigl[f(y)-f(x)\bigr]
\end{equation}
with $|\cdot|$ denoting the Euclidean distance. Throughout we will take the potential $\xi=\xi^{(\epsilon)}$ random, defined on some probability space $(\Omega,\FF,\BbbP)$, with an $\epsilon$-dependent law satisfying one  or both of the  following requirements (depending on the context):

\begin{assumption}
\label{ass1}
For each~$\epsilon>0$, $\{\xi^{(\epsilon)}(x)\colon x\in D_\epsilon\}$ are independent with
 \begin{equation}
\label{E:1.4ua}
\exists K>1\vee d/2\colon\quad \sup_{\epsilon\in (0,1)}\,\max_{x\in D_\epsilon}\,
 \E\left(|\xi^{(\epsilon)}(x)|^K\right)<\infty.
 \end{equation}
Moreover, there is $U\in C_{\rm b}(D, \R)$ such that
\begin{equation}
\label{E:1.4a}
\E\xi^{(\epsilon)}(x)=U(x\epsilon),\qquad x\in D_\epsilon.
\end{equation}
\end{assumption}

\begin{assumption}
\label{ass2}
The bound \eqref{E:1.4ua} holds for some $K>2\vee d/2$. Moreover, there is $V\in C_{\rm b}(D, [0,\infty))$ such that 
\begin{equation}
 \text{\rm Var}\bigl(\xi^{(\epsilon)}(x)\bigr)=V(x\epsilon),\qquad x\in D_\epsilon.
\end{equation}
\end{assumption}

To ease our notations, we will often omit marking the 
$\epsilon$-dependence of~$\xi$. We are interested in the behavior of 
the eigenvalues $\lambda^{\ssup 1}_{D_\epsilon,\xi}< \lambda^{\ssup 2}_{D_\epsilon,\xi}\leq \lambda^{\ssup 3}_{D_\epsilon,\xi}\leq \dots $ of $H_{D_\epsilon,\xi}$ in the limit as $\epsilon\downarrow 0$.

Let $\Delta$ denote the continuum Laplacian with Dirichlet boundary conditions outside~$D$. As it turns out, the continuum (homogenized) counterpart of $H_{D_\epsilon,\xi}$ is the operator
\begin{equation}
\label{E:1.7}
H_{D,U}:=-\Delta+U(x)
\end{equation}
acting on the space $\cmss H_0^1(D)$ := closure of $C^\infty_0(D)$ in the norm $[\Vert f\Vert_{L^2(D)}^2+\Vert\nabla f\Vert_{L^2(D)}^2]^{1/2}$, where~$\nabla$ denotes the continuum gradient.
The operator $H_{D,U}$ is self-adjoint and, thanks to our conditions on~$D$ and~$U$, of compact resolvent. In particular, its spectrum 
is real-valued and discrete with no eigenvalue more than finitely degenerate --- we will thus write $\lambda_D^{\ssup k}$ to denote the $k$-th smallest eigenvalue of $H_{D,U}$. 
Our first conclusion is as follows:

\begin{theorem}
\label{thm1.1}
Under Assumption~\ref{ass1}, for each~$k\in\N$,
\begin{equation}
\label{E:1.4}
\lambda^{\ssup k}_{D_\epsilon,\xi}\,
\underset{\epsilon\downarrow0}{\overset{\BbbP}{\longrightarrow}}\,
\lambda_D^{\ssup k}.
\end{equation}
\end{theorem}

\begin{remark}
As we will show in the Appendix, the moment condition \eqref{E:1.4ua} is more or less optimal for \eqref{E:1.4} to hold. More precisely, if the negative part of~$\xi$ fails to have $d/2$-nd moment in~$d\ge3$, we get $\lambda^{\ssup k}_{D_\epsilon,\xi}\to-\infty$ as $\epsilon\downarrow0$. We expect (although have not addressed mathematically) this to be a result of appearance of \emph{localized} states. 
\end{remark}

The formula \eqref{E:1.4} determines the leading-order deterministic behavior
of the spectrum of~$H_{D_\epsilon,\xi}$. The control of the subleading orders (or even an expansion in powers of~$\epsilon$) is a challenging task which we will not tackle here. We will content ourself with a description of the asymptotic behavior of the leading \emph{random} correction. For reasons to be explained later, we will do this only for any collection of (asymptotically) simple eigenvalues. In order to state the result, we need to fix $\kappa\in (d/K,2\wedge d/2)$ and define the truncated potential
\begin{equation}
\label{bar-xi}
 \overline{\xi}(x):=\xi(x) 1_{\{|\xi(x)|\le \epsilon^{-\kappa}\}}.
\end{equation}
Our second main result is then:

\begin{theorem}
\label{thm1.2}
Suppose Assumptions~\ref{ass1}--\ref{ass2} hold, fix $n\in\N$ and let 
$k_1,\dots,k_n\in\N$ be distinct indices such that the 
eigenvalues $\lambda_D^{\ssup{k_1}},\dots,\lambda_D^{\ssup{k_n}}$ 
of $H_{D,U}$ are simple. Then, in the limit as 
$\epsilon\downarrow0$, the law of the random vector
\begin{equation}
\Biggl(\frac{\lambda_{D_\epsilon,\xi}^{\ssup{k_1}}-
\E \lambda_{D_\epsilon,{\overline{\xi}}}^{\ssup{k_1}}}{\epsilon^{d/2}},\dots,
\frac{\lambda_{D_\epsilon,\xi}^{\ssup{k_n}}-
\E \lambda_{D_\epsilon,{\overline{\xi}}}^{\ssup{k_n}}}{\epsilon^{d/2}}\Biggr)
\label{multivariate}
\end{equation}
tends weakly to a multivariate normal with mean zero and covariance 
matrix $\sigma_D^2=\{\sigma_{ij}^2\}_{i,j=1}^n$ given by
\begin{equation}
\label{cov}
\sigma^2_{ij}:=\int_D\varphi_D^{\ssup{k_i}}(x)^2
\varphi_D^{\ssup{k_j}}(x)^2\,V(x)\,\textd x,
\end{equation}
where $\{\varphi_D^{\ssup{k_i}}\colon i=1,\dots,n\}$ is a collection of $L^2$-normalized eigenfunctions of~$H_{D,U}$ for indices $k_1,\dots,k_n$ and~$V(x)$ is the function from 
\eqref{E:1.4a}.
\end{theorem}

We note that, for simple eigenvalues, the eigenfunctions are determined up to an overall sign (they can always be chosen real valued). In particular, all choices of the eigenfunctions lead to the same value of the integral \eqref{cov}. A deeper, albeit related, reason for excluding degenerate eigenvalues is the fact that we work directly with \emph{ordered} eigenvalues (and not, e.g., the resolvent or some other symmetric function thereof). We expect that, for degenerate eigenvalues, the individual fluctuations are still Gaussian but the order is decided by combining the fluctuation with the expected value (which we control only to the leading order). We do not find this restriction much of a loss as, for generic~$D$ and~$U$, all eigenvalues of~$H_{D,U}$ will be non-degenerate.

\begin{remark}
Under Assumption~\ref{ass1}, we will see in~\eqref{truncation} below that the truncation \eqref{bar-xi} has no effect, with probability tending to 1 as $\epsilon\downarrow 0$.  
However, it turns out that the truncation does affect the mean value 
$\E\lambda_{D_\epsilon,\xi}^{\ssup{1}}$ for small $K$, see again the Appendix. 
Therefore it is necessary to retain the truncated potential inside the 
expectations in~\eqref{multivariate}.
\end{remark}

We refer the reader to our earlier paper~\cite{BFK14} for a thorough discussion of the above problem as well as related references. We will only mention to papers where we feel an update is necessary. First, an earlier work of Bal~\cite{Bal08} derived very similar homogenization and 
fluctuation results for the eigenvalues of a continuum Anderson Hamiltonian. However, there are a number of important differences:  
\begin{enumerate}
\item the weak convergence in~\cite{Bal08} is proved around the homogenized eigenvalues
rather than mean values,
\item the results hold also for sufficiently fast mixing random potentials, 
\item the spatial dimension is assumed to be less than or equal to three, $d\le3$, and
\item stronger moment assumption than ours are required.
\end{enumerate}
In particular, if one applies the method of~\cite{Bal08} to discrete 
independent potentials, it requires boundedness of the fourth moments. 
We believe this is because we use a completely 
different, mostly probabilistic approach. 

Second, related results concerning the low-lying eigenvalues of a random Laplacian arising from random conductances have recently been obtained by Flegel, Haida and Slowik~\cite{FHS}. Also there homogenization of the individual eigenvalues to those of a continuum (albeit ``homogenized'') Laplacian is obtained under more or less optimal moment condition on the random conductances.

\subsection*{Notations}
\noindent
Let us collect the notations that will be needed throughout this work.
We write $\Vert f\Vert_{p}$ for the canonical $\ell^p$-norm 
of~$\R$- or~$\R^d$-valued functions $f$ on~$\Z^d$. When~$p=2$, we use 
$\langle f,h\rangle$ to denote the associated inner product in 
$\ell^2(\Z^d)$. All functions defined \emph{a priori} only 
on~$D_\epsilon$ will be regarded as extended by zero 
to~$\Z^d\smallsetminus D_\epsilon$.
In order to control convergence to the continuum problem, 
it will sometimes be convenient to work with the scaled $\ell^p$-norm,
\begin{equation}
\Vert f\Vert_{\epsilon,p}:=
\biggl(\epsilon^d\sum_{x\in \Z^d}|f(x)|^p\biggr)^{\ffrac1p}.
\end{equation}
For~$p=2$, we will write $\langle f,g\rangle_{\epsilon,2}$ to denote the 
inner product associated with~$\Vert\cdot\Vert_{\epsilon,2}$. For 
functions~$f,g$ of a continuum variable, we write the norms 
as~$\Vert f\Vert_{L^p(\R^d)}$ and the inner product in~$L^2(\R^d)$ 
as $\langle f,g\rangle_{L^2(\R^d)}$. 
The discrete gradient $\DG f(x)$ is defined as the vector in~$\R^d$ whose~$i$-th component is
$f(x+\hate_i)-f(x)$, where $\{\hate_i\}_{i=1}^d$ is the canonical
basis of $\R^d$. 

Some of our computations in the proofs below will require suitable block averaging.
For~$L\in\N$ and $x\in\Z^d$, let $B_L(x):=Lx+\{0,\ldots,L-1\}^d$ and
for any $f\colon\Z^d\to\R$, define
\begin{equation}
\label{local-average}
f_L(x):=\sum_{y\in\Z^d}1_{B_L(y)}(x)\sum_{z\in B_L(y)}L^{-d}f(z).
\end{equation}
Note that, for each given~$x$, exactly one~$y$ contributes to the first sum; the resulting function is then constant on square blocks of side~$L$ and it equals to the average of~$f$ on each of them.

Recall that we assumed~$D$ to be a bounded open set in $\R^d$ with $C^{1,\alpha}$-boundary for some $\alpha>0$. This ensures a corresponding level of regularity of the eigenfunction. Indeed, 
by, e.g.,~Corollary~8.36 of Gilbarg and Trudinger~\cite{GT}, the eigenfunctions $\varphi^{\ssup k}_D$ of $H_{D,U}$ obey
\begin{equation}
\label{E:3.6aa}
\varphi^{\ssup k}_D\in C^{1,\alpha}(\overline D),
\end{equation}
that is, they are continuously differentiable in $D$ with the gradient 
uniformly $\alpha$-H\"older continuous. (In particular, the integral \eqref{cov} is convergent.) Concerning the discrete problem, 
we denote by $g_{D_\epsilon,\xi}^{\ssup k}$ an (real-valued) eigenfunction of~$H_{D_\epsilon,\xi}$ normalized in $\ell^2(\Z^d)$; this is again 
determined up to a sign whenever the~$k$-th eigenvalue is non-degenerate.

Finally, throughout the paper $c$ denotes a constant depending only on 
$d, D, K$ and $k$ whose value may change from line to line. 
We write $\epsilon^{0-}$ ($\epsilon^{0+}$) for a negative 
(resp.~positive) power of $\epsilon$ for simplicity.


\section{Convergence to homogenized eigenvalues}
\nopagebreak\label{sec3}\noindent
We are now in a position to start the exposition of the proofs. 
Here we will prove Theorem~\ref{thm1.1} dealing with the 
convergence of the random eigenvalues to those of the continuum problem. 


\subsection{Truncation}
\noindent
As is common whenever unbounded random variables get involved, we will deal with large values of the potential via a suitable truncation. We begin by noting:

\begin{lemma}
\label{lemma-2.1}
Under Assumption~\ref{ass1}, for each $\kappa\in (d/K, d\wedge 2)$ we have
\begin{equation}
\label{truncation}
\BbbP\Bigl(\,
\max_{x\in D_\epsilon}|\xi(x)|> \epsilon^{-\kappa}\Bigr)\,\underset{\epsilon\downarrow0}\longrightarrow\,0.
\end{equation}
\end{lemma}

\begin{proofsect}{Proof}
This follows from a union bound, Chebyshev's inequality, the bound \eqref{E:1.4ua} and the fact that definition \eqref{E:1.1} implies that $|D_\epsilon|$ is order~$\epsilon^{-d}$.
\end{proofsect}

We henceforth fix a $\kappa\in (d/K, d\wedge 2)$ so that~\eqref{truncation} holds, pick~$r$ satisfying 
\begin{equation}
\label{E:r-eq}
1\vee d/2<r< d/\kappa
\end{equation}
and assume 
\begin{equation}
\label{xi-bounded}
 \max_{x\in D_\epsilon}|\xi(x)|\le \epsilon^{-\kappa}.
\end{equation}
This is tantamount to working with the truncated potential $\overline{\xi}$ in place of~$\xi$, which we will however ignore notationally; thanks Lemma~\ref{lemma-2.1}, it suffices to prove Theorem~\ref{thm1.1} under this additional assumption.

Given any choice of the normalized eigenfunctions $\{\varphi_D^{\ssup{j}}\}_{j\ge 1}$ of the operator \eqref{E:1.7}, for each 
$\gamma>0$ and each~$\epsilon\in(0,1)$ define the event 
\begin{equation}
E_{k,\epsilon,\gamma}:=\left\{\xi\colon
\begin{aligned}
&\max_{1\le j\le k}
\left|\langle \xi-U (\epsilon \cdot), 
\varphi_D^{\ssup{j}}(\epsilon \cdot)^2\rangle_{\epsilon,2}\right|<\gamma\\
&{\|\xi\|_{\epsilon,r}<4|D|\max_{x\in D_\epsilon}\E[|\xi(x)|^r]}
\end{aligned}\right\}.
\label{E_k}
\end{equation}

\begin{remark}
The constant 4 above plays no special role in the proof. Any larger constant would work as well. We will make use of this observation (only) in the proof of Lemma~\ref{unif-g} below.
\label{E_k-change}
\end{remark}

\noindent
Then we observe:
\begin{lemma}
\label{good-event}
Under Assumption~\ref{ass1} and~\eqref{xi-bounded}, for all $k\in\N$ and all $\gamma>0$, and all~$\epsilon>0$ sufficiently small, 
\begin{equation}
\BbbP(E_{k,\epsilon,\gamma}^\cc) \le \exp\{-\epsilon^{0-}\}\,.
\end{equation}
\end{lemma}

\begin{proofsect}{Proof}
The proof is based on a number of elementary concentration-of-measure arguments.
Let us fix $a_0<a_1<\cdots<a_N:=\kappa <d/r$ such that
\begin{equation}
 0<a_0<\frac{d}{2} \qquad\textrm{and}\qquad
\frac{a_{n-1}}{a_n}>\frac{1}{K},\quad n=1,\dots,N. 
\end{equation}
Using this sequence, we write
\begin{equation}
\begin{split}
\xi(x)-U (\epsilon x)
&=(\xi(x)-U (\epsilon x))1_{\{|\xi(x)|< \epsilon^{-a_0}\}}
+\sum_{n=1}^N(\xi(x)-U (\epsilon x))
1_{\{\epsilon^{-a_{n-1}}\le|\xi(x)|< \epsilon^{-a_n}\}}\\
&=:\eta(x)+\sum_{n=1}^N\zeta_n(x)
\end{split}
\end{equation}
so that 
\begin{equation}
\label{separation}
\begin{split}
&\BbbP\left(\left|\langle \xi-U (\epsilon \cdot), 
 \varphi_D^{\ssup{j}}(\epsilon \cdot)^2\rangle_{\epsilon,2}\right|
 \ge\gamma\right)\\
&\quad \le\BbbP\biggl(\biggl|\sum_{x\in D_\epsilon}\epsilon^d\eta(x)
 \varphi_D^{\ssup{j}}(\epsilon x)^2\biggr|
 \ge \frac{\gamma}{2}\biggr)+\sum_{n=1}^N
 \BbbP\biggl(\sum_{x\in D_\epsilon}\epsilon^d|\zeta_n(x)|
 \varphi_D^{\ssup{j}}(\epsilon x)^2
 \ge\frac{\gamma}{2N}\biggr).
\end{split}
\end{equation}
First, the Azuma-Hoeffding inequality shows 
\begin{equation}
\begin{aligned}
 \label{E:2.6}
 \BbbP\left(\left|\sum_{x\in D_\epsilon}
 \epsilon^d\eta(x)\varphi_D^{\ssup{j}}(\epsilon x)^2
 \right|\ge \frac{\gamma}{2}
 \right)
& \le 2\exp\bigl\{-c\epsilon^{-d+2a_0}\bigr\}\\
&\le \exp\bigl\{-\epsilon^{0-}\bigr\}
\end{aligned}
\end{equation}
for all sufficiently small $\epsilon$. 
Note that due to our use of the truncated potential, a proper use of Azuma-Hoeffding requires an additional intermediate step reflecting on the fact that $\E[\eta(x)]$ may not be zero. 
This is handled by replacing~$\gamma/2$ above with $\gamma/4$ and noting that the difference $\E[\eta(x)]$ converges to zero uniformly in $x$. Our implicit truncation~\eqref{xi-bounded} also sometimes requires this type of considerations and they will be done implicitly in what follows. 

Next, we deal with the second term in~\eqref{separation}. When $\epsilon$ is sufficiently small, we can bound each summand by
\begin{equation}
\label{A.6} 
\BbbP\biggl(\sum_{x\in D_\epsilon}\epsilon^d|\zeta_n(x)|
 \ge \frac{\gamma}{2N\|\varphi_D^{\ssup{j}}\|_{\infty}^2}\biggr)
\le \BbbP\biggl(\sum_{x\in D_\epsilon}1_{\{\zeta_n(x)\neq 0\}}
\ge\epsilon^{-d+a_n}\frac{\gamma}{4N\|\varphi_D^{\ssup{j}}\|_{\infty}^2}\biggr).
\end{equation}
Since $\{1_{\{\zeta_n(x)\neq 0\}}\}_{x\in D_\epsilon}$ are stochastically dominated by independent Bernoulli variables with success probability
\begin{equation}
\BbbP(\zeta_n(x)\ne0)\le
\BbbP\bigl(|\xi(x)|>\epsilon^{-a_{n-1}}\bigr)
\le \epsilon^{a_{n-1}K}\sup_{\epsilon\in(0,1)}
\sup_{x\in D_\epsilon}\E[|\xi(x)|^K]
\end{equation}
and $a_{n-1}K>a_n$, a simple application of the Bernstein 
inequality tells us that the right-hand side of~\eqref{A.6}
is bounded by $\exp\{-\epsilon^{0-}\}$ for sufficiently small 
$\epsilon$.

The argument for $\|\xi\|_{\epsilon,r}$ is almost the same. We write $M:=|D|\max_{x\in D_\epsilon}\E[|\xi(x)|^r]$ and, using the above sequence,
\begin{equation}
\begin{split}
|\xi(x)|^r
&=|\xi(x)|^r1_{\{|\xi(x)|< \epsilon^{-a_0}\}}
+\sum_{n=1}^N|\xi(x)|^r
1_{\{\epsilon^{-a_{n-1}}\le|\xi(x)|< \epsilon^{-a_n}\}}\\
&=:\eta(x)+\sum_{n=1}^N\zeta_n(x)
\end{split}
\end{equation}
so that 
\begin{equation}
 \BbbP\biggl(\sum_{x\in D_\epsilon}\epsilon^d|\xi(x)|^r
 \ge {4M}\biggr)
\le\BbbP\biggl(\sum_{x\in D_\epsilon}\epsilon^d\eta(x)
 \ge{3M}\biggr)+\sum_{n=1}^N
 \BbbP\biggl(\sum_{x\in D_\epsilon}\epsilon^d\zeta_n(x)
 \ge\frac{{M}}{N}\biggr).
\end{equation}
When $\epsilon$ is sufficiently small, we have
\begin{equation}
 \sum_{x\in D_\epsilon}\epsilon^d\E[\eta(x)]\le 2M
\end{equation}
and we can again appeal to the Azuma-Hoeffding inequality to get
\begin{equation}
\begin{split}
\BbbP\biggl(\sum_{x\in D_\epsilon}\epsilon^d\eta(x)
 \ge{3M}\biggr)
&{\le \BbbP\biggl(\sum_{x\in D_\epsilon}\epsilon^d\left(\eta(x)-\E[\eta(x)]
\right) \ge M\biggr)}\\
&\le 2\exp\bigl\{-c\epsilon^{-d+2a_0}\bigr\}.
\end{split}
\end{equation}
The rest of the argument is very similar to above and we omit further details. 
\end{proofsect}

\subsection{Upper bound by homogenized eigenvalue}
\noindent
We will now prove the upper bound in Theorem~\ref{thm1.1}.
Instead of individual eigenvalues, we will work with their sums
\begin{equation}
\label{E:3.12ua}
\Lambda_k^{\epsilon}(\xi):=\sum_{i=1}^k\lambda^{\ssup i}_{D_\epsilon,\xi}
\quad\text{and}\quad
\Lambda_k:=\sum_{i=1}^k\lambda^{\ssup i}_{D}.
\end{equation}
These quantities are better suited for dealing with degeneracy because 
they admit a variational 
characterization (a.k.a.~the Ky Fan Maximum
Principle~{KyFan}) of the form
\begin{equation}
\label{E:3.22}
\Lambda_k^{\epsilon}(\xi)=\,\inf_{\begin{subarray}{c}
h_1,\dots,h_k\\\textrm{ONS}
\end{subarray}}\,\,
\sum_{i=1}^k\bigl(\epsilon^{-2}\|\nabla^{\ssup{\textd}}h_i\|_2^2
+\langle\xi, h_i^2\rangle\bigr)
\end{equation}
and
\begin{equation}
\label{E:3.22b}
\Lambda_k=\,\inf_{\begin{subarray}{c}
\psi_1,\dots,\psi_k\\\textrm{ONS}
\end{subarray}}\,\,
\sum_{i=1}^k\bigl(\,\|\nabla\psi_i\|_{L^2(\R^d)}^2+\langle U, 
\psi_i^2\rangle_{L^2(\R^d)}\bigr),
\end{equation}
where the acronym ``ONS'' imposes that the $k$-tuple of functions (all assumed in the domain of the gradient in the latter case) forms
an orthonormal system in the subspace corresponding to Dirichlet boundary
conditions. 

The infima in \twoeqref{E:3.22}{E:3.22b} are both achieved by a collection of lowest-$k$ eigenfunctions of operators $H_{D_\epsilon,\xi}$, resp.,~$H_{D,U}$. This offers a strategy for comparing the two quantities: Take the eigenfunctions of one problem and use them, after discretizing/undiscretizing, as trial functions in the other variational problem. Starting from the continuum problem, this strategy is relatively easy to implement as attested by:

\begin{proposition}
\label{lemma-3.4}
For any $k\in \N$ and any $\gamma>0$,
\begin{equation}
\label{E:3.15}
E_{k,\epsilon,\gamma}\subseteq\bigl\{\Lambda_k^{\epsilon}(\xi)\le\Lambda_k+3\gamma\bigr\}
\end{equation}
holds for all sufficiently small $\epsilon>0$. 
In particular, under Assumption~\ref{ass1}, for any~$\delta>0$,
\begin{equation}
\lim_{\epsilon\downarrow0}\BbbP
\bigl(\Lambda_k^{\epsilon}(\xi)\le\Lambda_k+\delta\bigr)=1.
\end{equation}
\end{proposition}

\begin{proofsect}{Proof}
Consider (a choice of) an ONS of the first~$k$ eigenfunctions 
$\varphi_D^{\ssup 1},\dots,\varphi_D^{\ssup k}$ of~$H_{D,U}$. 
Recall that all of these are in $C^{1,\alpha}(\overline{D})$. 
Now define
\begin{equation}
\label{E:3.32}
f_i(x):=
\begin{cases}
\varphi_D^{\ssup i}(x\epsilon),\qquad&\text{if }x\in D_\epsilon,
\\
0,\qquad&\text{otherwise}.
\end{cases}
\end{equation}
Thanks to uniform continuity of the eigenfunctions, we then have
\begin{equation}
\langle f_i,f_j\rangle_{\epsilon,2}\,\underset{\epsilon\downarrow0}
\longrightarrow\,\langle \varphi_D^{\ssup i},\varphi_D^{\ssup j}
\rangle_{L^2(D)}=\delta_{ij}
\end{equation}
and so for~$\epsilon$ small the functions $f_1,\dots,f_k$ are nearly
 mutually orthogonal. Applying the Gram-Schmidt orthogonalization
 procedure, we conclude that there are functions
$h_1^{\epsilon},\dots, h_k^{\epsilon}$ and coefficients
$a_{ij}(\epsilon)$, $1\le i,j\le k$, such that
\begin{equation}
\label{E:3.35}
h_i^\epsilon=\sum_{j=1}^k\bigl(\delta_{ij}+a_{ij}(\epsilon)\bigr)f_j,\qquad i=1,\dots,k,
\end{equation}
with
\begin{equation}
\label{E:3.36}
\langle h_i^\epsilon,h_j^\epsilon\rangle_{\epsilon,2}=\delta_{ij}
\quad\text{and}\quad\max_{i,j}|a_{ij}(\epsilon)|\,
\underset{\epsilon\downarrow0}\longrightarrow\,0.
\end{equation}
Moreover, the definition of~$f_i$ and the $C^{1,\alpha}$-regularity 
of the eigenfunctions imply
\begin{equation}
\label{E:3.33}
\sup_{\begin{subarray}{c}
y\in D\\\dist_\infty(y,D^\cc)>2\epsilon
\end{subarray}}
\Bigl|\nabla\varphi_D^{\ssup i}(y)-\epsilon^{-1}(\nabla^{\ssup{\textd}}f_i)(\lfloor y/\epsilon\rfloor)\Bigr|\,\underset{\epsilon\downarrow0}\longrightarrow\,0
\end{equation}
and the same applies to $h_i^\epsilon$ instead of $f_i$ as well. 
Since $\nabla\varphi_D^{\ssup i}$ and $\epsilon^{-1}(\nabla^{\ssup{\textd}}f_i)$ are also bounded, we thus get
\begin{equation}
\epsilon^{-1}\Vert\nabla^{\ssup{\textd}} h_i^\epsilon\Vert_{\epsilon,2}\,\underset{\epsilon\downarrow0}\longrightarrow\,
\Vert\nabla\varphi_D^{\ssup i}\Vert_{L^2(\R^d)}\,.
\end{equation}
The continuity of~$U$ shows that, also
\begin{equation}
\label{E:3.23}
\bigl\langle U(\epsilon\cdot),(h_i^\epsilon)^2\bigr\rangle_{\epsilon,2}
\,\underset{\epsilon\downarrow0}\longrightarrow\,
\bigl\langle U,(\varphi_D^{\ssup i})^2\bigr\rangle_{L^2(\R^d)}.
\end{equation}
Therefore, given any~$\gamma>0$, as soon as $\epsilon>0$ is sufficiently small (independent of~$\xi$) the
variational characterization \eqref{E:3.22} yields
\begin{equation}
\Lambda_k^{\epsilon}(\xi)\le\Lambda_k+\gamma+\sum_{i=1}^k
\bigl\langle \xi-U(\epsilon\cdot),(h_i^\epsilon)^2\bigr\rangle_{\epsilon,2}.
\end{equation}
The summands on the right-hand side are bounded as
\begin{equation}
\begin{aligned}
&\left\vert\bigl\langle 
\xi-U(\epsilon\cdot),(h_i^\epsilon)^2\bigr\rangle_{\epsilon,2}\right\vert\\
&\quad\le \left\vert\bigl\langle 
\xi-U(\epsilon\cdot),f_i^2\bigr\rangle_{\epsilon,2}\right\vert
+\bigl(\max_{i,j=1,\dots,k}
|a_{ij}(\epsilon)|\bigr)\,\bigl(\max_{\ell=1,\dots,k}\Vert \varphi_D^{\ssup \ell}\Vert_{\infty}^2\bigr)
\bigl(\|\xi\|_{\epsilon,1}+\|U(\epsilon\cdot)\|_{\epsilon,1}\bigr).
\end{aligned}
\end{equation}
Noting that the first term is at most~$\gamma$ and $\|\xi\|_{\epsilon,1}$ is bounded on $E_{k,\epsilon,\gamma}$, this will be less than~$2\gamma$ as soon as~$\epsilon$ is sufficiently small (again, independent of~$\xi$).
\end{proofsect}

\begin{corollary}
\label{cor-3.4}
For each~$k\in\N$ and each~$\gamma>0$ there is $c_{k,\gamma}$ such that for all~$\epsilon\in(0,1)$,
\begin{equation}
E_{k,\epsilon,\gamma}\subseteq\bigl\{\Lambda_k^\epsilon(\xi)\le c_{k,\gamma}\bigr\}
\end{equation}
\end{corollary}

\begin{proofsect}{Proof}
For small-enough~$\epsilon$, this follows from \eqref{E:3.15} and the fact that~$\Lambda_k$ is deterministic. In the complementary range of~$\epsilon\in(0,1)$, we note that \eqref{xi-bounded} gives $\langle \xi,(h_i)^2\rangle\le\epsilon^{-\kappa}$ for each~$i=1,\dots,k$. This reduces the problem to bounding the sum of the first~$k$ eigenvalues of $\epsilon^{-2}$-times the (negative) Dirichlet Laplacian in square-domains of side-length proportional to~$\epsilon^{-1}$, for which the spectrum is explicitly computable (and the eigenvalues are bounded uniformly in~$\epsilon$). 
\end{proofsect}

\subsection{Elliptic regularity for eigenfunctions}
\label{regularity}\noindent
For the corresponding lower bound of~$\Lambda_k^\epsilon$ by~$\Lambda_k$, we will start with the collection of the eigenfunctions of~$H_{D_\epsilon,\xi}$ and turn these into functions over the continuum domain~$D$. The main technical obstacle is that the discrete eigenfunctions are \emph{random} and so the derivation of the needed regularity estimates (which for the upper bound were supplied by the fact that the eigenfunctions of~$H_{D,U}$ are $C^{1,\alpha}$) require a non-trivial use of elliptic regularity theory. As usual, a starting point for these is a suitable functional inequality:

\begin{lemma}[Sobolev inequality]
\label{lemma-Sobolev}
Let~$q\in[2,\infty)$ obey $q < 2d/(d-2)$ in~$d\ge3$. Then there 
is $c(D,q)>0$ such that
\begin{equation}
\label{E:3.10a}
\epsilon^{-2}\bigl\Vert\nabla^{\ssup{\textd}}f\bigr\Vert_{\epsilon,2}^2
+\bigr\Vert f\bigr\Vert_{\epsilon,2}^2
\ge c(D,q)\bigl\Vert f\bigr\Vert_{\epsilon,q}^2
\end{equation}
holds for all $\epsilon\in(0,1)$ and all~$f\colon\Z^d\to\R$ 
with~$\supp f\subseteq D_\epsilon$.
\end{lemma}

Although this is quite standard, we provide a (short) proof in the Appendix (this will also make it clear that our normalizations are legitimate).
A considerably deeper use of elliptic regularity theory is required to control the individual eigenfunctions of~$H_{D_\epsilon,\xi}$. In order to state our first such estimate, pick $\rho\in (0,1-\kappa r/d)$, where~$r$ is as in \eqref{E:r-eq}, set $L:=\epsilon^{-\rho}$ and, recalling the definition of block-averaged function~\eqref{local-average}, define
\begin{equation}
\overline{\xi}_L(x):=\bigl(U(\epsilon\cdot)-\xi(\cdot)\bigr)_L(x)
\end{equation}
Consider the event
\begin{equation}
\label{F2}
F_{\epsilon,\gamma}:=
\bigl\{\xi\colon\|\overline{\xi}_L\|_{\epsilon,r}<\gamma \bigr\}.
\end{equation}
Then we have: 

\begin{proposition}
\label{prop2}
Suppose Assumption~\ref{ass1}. For all $p>1$, all $k\in\N$, and any choice of 
the $k$-th eigenfunction $g_{D_\epsilon,\xi}^{\ssup k}$ of~$H_{D_\epsilon,\xi}$, we have
\begin{equation}
\label{E:3.4}
\sup_{0<\epsilon<1}\,\,\sup_{\xi\in E_{k,\epsilon,\gamma}\cap F_{\epsilon,\gamma}}\,\,
\bigl\Vert \epsilon^{-d/2}g_{D_{\epsilon},\xi}^{\ssup k}
\bigr\Vert_{\epsilon,p}<\infty
\end{equation}
uniformly in sufficiently small~$\gamma>0$.
\end{proposition}

\begin{remark}
\label{remark-excuse}
In Lemma~\ref{good-event} we showed that~$E_{k,\epsilon,\gamma}$ will occur with overwhelming probability for small enough~$\gamma$ and~$\epsilon$, and a similar statement will be shown for $F_{\epsilon,\gamma}$ in Lemma~\ref{good-event2}. The reason why event $F_{\epsilon,\gamma}$ needs to be included in the statement above is that it ensures, via Proposition~\ref{lemma3.5} with~$k=1$ below, a lower bound on the principal eigenvalue (uniform in~$\xi\in E_{k,\epsilon,\gamma}\cap F_{\epsilon,\gamma}$). Combining with Corollary~\ref{cor-3.4} we then get an upper bound on the individual eigenvalues for each $k\ge 2$, which then feeds into the proof of \eqref{E:3.4} for~$k\ge2$. Since, for $k=1$, Corollary~\ref{cor-3.4} bounds the principal eigenvalue directly, the inclusion of event $F_{\epsilon,\gamma}$ in \eqref{E:3.4} is redundant and no logical conflict arises. 
\end{remark}

\begin{proofsect}{Proof of Proposition~\ref{prop2}}
The proof is based on the Moser iteration scheme for solutions of 
elliptic PDEs. This technique needs to be adapted to the discrete 
setting which has fortunately already been done in a 
recent paper of Andres, Deuschel and Slowik~\cite{ADS} on
homogenization of the random conductance model with general ergodic 
random conductances subject (only) to suitable moment conditions. We cite both notation and conclusions at liberty 
from there. 

Given~$s\ge 1$, let us write $a^{[s]}:=|a|^s{\rm sign}(a)$ for the 
signed-power function and $f^{[s]}(x)$ for $(f(x))^{[s]}$.
By equation (40) of~\cite{ADS}, there is a constant $c(s)$ depending only on~$s$ such that for any function $\phi\colon \Z^d\to\R$ with finite support
\begin{equation}
\sum_{x\in \Z^d}\bigl|\nabla^{\ssup{\textd}}
\phi^{[s]}(x)\bigr|^2
\le c(s)\sum_{x\in \Z^d}\sum_{i=1}^d
\bigl(\phi^{[s-1]}(x)+\phi^{[s-1]}(x+\hate_i)\bigr)^2
\bigl|\nabla^{\ssup{\textd}}_i \phi(x)\bigr|^2,
\label{MI1}
\end{equation}
where $\nabla^{\ssup{\textd}}_i$ is the $i$-th component of the discrete 
gradient.
We further use equation (42) of~\cite{ADS} --- with the specific choices 
$\alpha:=2s-2$ and~$\beta:=1$
--- to get
\begin{equation}
\begin{aligned}
\bigl(\phi^{[s-1]}(x)+\phi^{[s-1]}(x+\hate_i)\bigr)^2\bigl|
\nabla^{\ssup{\textd}}_i \phi(x)\bigr|
&\le 2\bigl(|\phi(x)|^{2s-2}+|\phi(x+\hate_i)|^{2s-2}\bigr)
 |\nabla^{\ssup{\textd}}_i \phi(x)|\\
&\le 2\bigl|\nabla^{\ssup{\textd}}_i \phi^{[2s-1]}(x)\bigr|. \label{MI2}
\end{aligned}
\end{equation}
The key point of using the signed-power function is that 
$\nabla^{\ssup{\textd}}_i \phi(x)$ and 
$\nabla^{\ssup{\textd}}_i \phi^{[2s-1]}(x)$ are of the same sign. 
This permits us to wrap \eqref{MI1} as
\begin{equation}
\begin{aligned}
\sum_{x\in \Z^d}\bigl|\nabla^{\ssup{\textd}}\phi^{[s]}(x)\bigr|^2
&\le 2c(s)\sum_{x\in D_\epsilon}\sum_{i}
|\nabla^{\ssup{\textd}}_i \phi^{[2s-1]}(x)||\nabla^{\ssup{\textd}}_i \phi(x)|\\
&= 2c(s)\,
\bigl\langle\nabla^{\ssup{\textd}} \phi^{[2s-1]},\nabla^{\ssup{\textd}} 
\phi\bigr\rangle.
\end{aligned}
\label{E:3.7}
\end{equation}
where we recall that the brackets stand for the usual inner product in~$\ell^2(\Z^d)$. 

Now let us assume that $\phi$ solves the equation
 $(-\epsilon^{-2}\Deltad+\xi)\phi=\lambda\phi$ in~$D_\epsilon$ and
vanishes outside $D_\epsilon$. Then we have 
\begin{equation}
\label{MI2.5}
\epsilon^d\bigl\langle\nabla^{\ssup{\textd}} \phi^{[2s-1]},
\nabla^{\ssup{\textd}} \phi\bigr\rangle=\epsilon^d\bigl\langle\phi^{[2s-1]},-\Deltad 
\phi\bigr\rangle
=\epsilon^{2+d}\bigl\langle\phi^{[2s-1]},(\lambda-\xi) 
\phi\bigr\rangle.
\end{equation}
Since $\phi^{[2s-1]}$ and~$\phi$ have the same sign,
the right-hand side is bounded by
\begin{equation}
\begin{split}
\epsilon^{2+d}\bigl\langle|\phi|^{2s},({\lambda_+}-\xi) \bigr\rangle
&\le \epsilon^{2}\left(\sum_{x\in D_\epsilon}
\epsilon^d|{\lambda_+}-\xi(x)|^{r}\right)^{1/r}
\left(\sum_{x\in D_\epsilon}\epsilon^{d}
\vert\phi(x)\vert^{2sr'}\right)^{1/r'}\\
&= \epsilon^{2}\|{\lambda_+}-\xi\|_{\epsilon, r}\,
\|\phi\|^{2s}_{\epsilon, 2sr'},
\label{MI3}
\end{split} 
\end{equation}
where $\lambda_+$ stands for the positive part of $\lambda$ and $r'$ is 
the H\"older conjugate of~$r$.
On the other hand, by Lemma~\ref{lemma-Sobolev}, for any~$q$ satisfying 
$2\le q<2d/(d-2)$ (with the right-hand inequality dropped in~$d=1,2$) we have
\begin{equation}
\sum_{x\in D_\epsilon}\epsilon^d
\bigl(\epsilon^{-2}|\nabla^{\ssup{\textd}} 
\phi^{[s]}(x)|^2+|\phi^{[s]}(x)|^2\bigr) 
\ge c(D,q)\left(\sum_{x\in D_\epsilon}
\epsilon^d\bigl|\phi^{[s]}(x)\bigr|^q\right)^{2/q}, 
\label{MI4}
\end{equation}
for some constant~$c(D,q)>0$.
The right-hand side is a multiple of~$\|\phi\|_{\epsilon,sq}^{2s}$ 
while, in light of \twoeqref{E:3.7}{MI3}, the left-hand side is bounded 
by a term involving~$\|\phi\|_{\epsilon,2sr'}^{2s}$.
This turns \eqref{MI4} into a recursion relation
\begin{equation}
\label{E:3.10}
\|\phi\|_{\epsilon,sq}\le 
\hat c\|\phi\|_{\epsilon,2sr'}
\end{equation}
for $\hat c:=[2c(s)c(D,q)^{-1}({\lambda_+}+\|\xi\|_{\epsilon,r})]^{\frac1{2s}}$. For~$r$ as in \eqref{E:r-eq} we get $r'<d/(d-2)$ in~$d\ge3$ and so, in all~$d\ge1$, we can find $q$ with~$2r'<q<2d/(d-2)$ and get an improvement in regularity.

Now pick~$s>1$ and let $\phi(x):=\epsilon^{-d/2}g_{D_\epsilon,\xi}^{\ssup k}(x)$ and~$\lambda:=\lambda_{D_\epsilon,\xi}^{\ssup k}$ and invoke the argument alluded to in Remark~\ref{remark-excuse}: 
For $k=1$, both $\|\xi\|_{\epsilon,r}$ and $(\lambda_{D_\epsilon,\xi}^{\ssup 1})_+$ are bounded on $E_{k,\epsilon,\gamma}$ uniformly in~$\epsilon$ by definition and Corollary~\ref{cor-3.4}, and so~$\hat c$ is bounded by an absolute constant. 
Moreover, $\Vert\phi\Vert_{\epsilon,2}=1$ by definition and, since $sr'\in(1,sq/2)$, for~$\tilde\alpha\in(0,1)$ such that~$2\tilde\alpha+sq(1-\tilde\alpha)=2sr'$, H\"older's inequality yields
\begin{equation}
\|\phi\|_{\epsilon,2sr'}\le \|\phi\|_{\epsilon,2}^{\tilde\alpha}\|\phi\|_{\epsilon,sq}^{1-\tilde\alpha}
\le\hat c^{1-\tilde\alpha}\|\phi\|_{\epsilon,2}^{\tilde\alpha}\|\phi\|_{\epsilon,2sr'}^{1-\tilde\alpha}\,,
\end{equation}
where the second inequality follows from \eqref{E:3.10}. This bounds~$\|\phi\|_{\epsilon,2sr'}$ by $\hat c^{\tilde\alpha^{-1}-1}$; an iterative use of \eqref{E:3.10} then yields \eqref{E:3.4}, as desired. 

For $k\ge 2$, we first use the conclusion for $k=1$ to complete the proof of Proposition~\ref{lemma3.5}, which shows that $\lambda_{D_\epsilon,\xi}^{\ssup 1}$ is bounded from below on $E_{k,\epsilon,\gamma}\cap F_{\epsilon,\gamma}$. Then combining with Corollary~\ref{cor-3.4}, we obtain the boundedness of $(\lambda_{D_\epsilon,\xi}^{\ssup k})_+$ on $E_{k,\epsilon,\gamma}\cap F_{\epsilon,\gamma}$ and the rest of the computation is the same as before.
\end{proofsect}

As a corollary, we get a regularity result for gradients of eigenfunctions as well: 

\begin{corollary}
\label{lemma-regularity}
Under Assumption~\ref{ass1}, for all $k\in\N$, and any choice of the $k$-th eigenfunction $g_{D_\epsilon,\xi}^{\ssup k}$ of~$H_{D_\epsilon,\xi}$,
\begin{equation}
\label{E:3.4aa}
\sup_{0<\epsilon<1}\,\sup_{\xi\in E_{k,\epsilon,\gamma}\cap F_{\epsilon,\gamma}}
\epsilon^{-2}
\bigl\|\nabla^{\ssup{\textd}}
g_{D_\epsilon,\xi}^{\ssup k}\bigr\|_2^2<\infty,
\end{equation}
uniformly in ~$\gamma\in(0,1)$.
\end{corollary}

\begin{proofsect}{Proof}
Just plug \eqref{E:3.4} in~\twoeqref{E:3.7}{MI3} with $s:=1$.
\end{proofsect}

Our final regularity lemma addresses approximations of functions by their piecewise-constant counterparts. 
Recall the definition of~$f_L$ from \eqref{local-average}. Then we have:

\begin{lemma}
\label{lemma-3.4new}
There is $C(d)<\infty$ such that, for any 
$p\in (1,2)$, any $L\in\N$ and any $f\colon \Z^d\to\R$ with finite support,
\begin{equation}
\Vert f^2-f_L^2\Vert_p
<C(d)L\Vert\nabla^{\ssup{\textd}}f\Vert_2\Vert f\Vert_{\frac{2p}{2-p}}.
\end{equation}
\end{lemma}

\begin{proofsect}{Proof}
For any $1\le p<2$, H\"older's inequality shows
\begin{equation}
\Vert f^2-f_L^2\Vert_p \le \Vert f-f_L\Vert_2
\Vert f+f_L\Vert_{\frac{2p}{2-p}}\,.
\end{equation}
The first term on the right is bounded by 
$cL\Vert\nabla^{\ssup{\textd}}f\Vert_2$
due to the Poincar\'e inequality and our definition of~$f_L$, while 
the second terms is at most $2\|f\|_{\frac{2p}{2-p}}$
since $f\mapsto f_L$ is a contraction. 
\end{proofsect}

\subsection{Lower bound by homogenized eigenvalue}
\noindent
We are now ready to tackle the lower bound in Theorem~\ref{thm1.1}. We start by showing that the event $F_{\epsilon,\gamma}$ from \eqref{F2} occurs with overwhelming probability when~$\epsilon$ is sufficiently small:

\begin{lemma}
\label{good-event2}
Under Assumption~\ref{ass1} and~\eqref{xi-bounded}, for any $\gamma>0$ and all~$\epsilon>0$ sufficiently small, 
\begin{equation}
\BbbP(F_{\epsilon,\gamma}^\cc)
\le \exp\{-\epsilon^{0-}\}.
\end{equation}
\end{lemma}

\begin{proofsect}{Proof}
Recall that~$L:=\epsilon^{-\rho}$ for $\rho\in (0,1-\kappa r/d)$ with~$r$ as in \eqref{E:r-eq}. 
Introducing
\begin{equation}
\Xi_L(y):=\sum_{y\in L\Z^d}(\epsilon L)^d
 \left|\sum_{z\in B_L(y)}L^{-d}\bigl(U(z\epsilon)-\xi(z)\bigr)\right|^r
\end{equation}
we may write
\begin{equation}
\label{E:2.46}
\|\overline{\xi}_L\|_{\epsilon,r}^r
 =\sum_{y\in L\Z^d}(\epsilon L)^d\Xi_L(y).
\end{equation}
Note that $(\epsilon L)^d$ is the reciprocal of the 
number of $y$'s with $\Xi_L(y)\neq 0$
up to a multiplicative constant. In addition, note also that
$\lim_{\epsilon\downarrow 0}\Xi_L(y)=0$ in probability for each~$y\in\Z^d$ (by the Law of Large Numbers and the fact that the truncated-field expectations converge to~$U$),  
$\sup_y\Xi_L(y) \le 2\epsilon^{-\kappa r}$ by \eqref{xi-bounded} and
\begin{equation}
 \sup_{\epsilon\in (0,1)}\,\,\sup_{y\in\Z^d}\,\,\E\left[\Xi_L(y)^{K/r}\right]
\le L^{-d}\sum_{z\in B_L(y)}
\E\left[|U(z\epsilon)-\xi(z)|^K\right]<\infty
\end{equation}
by Assumption~\ref{ass1}. 
Given these inputs, we will now prove 
\begin{equation}
\label{E:2.48}
 \BbbP\left(\,\sum_{y\in L\Z^d}(\epsilon L)^d\Xi_L(y)>\gamma\right)
\le \exp\{-\epsilon^{0-}\}
\end{equation}
for sufficiently small $\epsilon>0$, which by \eqref{E:2.46} (and redefinition of~$\gamma$) yields the desired claim. 

To get \eqref{E:2.48}, we proceed very much in the same way as in the proof
of Lemma~\ref{good-event}. 
For~$r$ and~$\rho$ as above, fix real numbers $a_0<a_1<\cdots<a_JL:=\kappa r<d$ satisfying
\begin{equation}
 0<a_0<\frac{d(1-\rho)}{2} \quad\textrm{and}\quad
\frac{a_{j-1}}{a_j}>\frac{r}{K}
\end{equation}
and write
\begin{equation}
\begin{split}
\Xi_L(y)
&=\Xi_L(y)1_{\{\Xi_L(y)< \epsilon^{-a_0}\}}
+\sum_{j=1}^J\Xi_L(y)
1_{\{\epsilon^{-a_{j-1}}\le\Xi_L(y)< \epsilon^{-a_j}\}}\\
&=:\eta(y)+\sum_{j=1}^J\zeta_j(y).
\end{split}
\end{equation}
The union bound then shows 
\begin{equation}
\BbbP\biggl(\sum_{y\in L\Z^d}\epsilon^d|\Xi_L(y)|^r
\ge\gamma\biggr)
\le\BbbP\biggl(\sum_{y\in L\Z^d}(\epsilon L)^d\eta(y)
\ge\frac{\gamma}{2}\biggr)+\sum_{j=1}^J
\BbbP\biggl(\sum_{y\in L\Z^d}(\epsilon L)^d\zeta_j(y)
\ge\frac{\gamma}{2J}\biggr).
\end{equation}
Since the above ``inputs'' yield $\sup_y\E[\eta(y)]=o(1)$ as 
$\epsilon\downarrow0$, 
the Azuma-Hoeffding inequality implies
\begin{equation}
\BbbP\biggl(\sum_{y\in L\Z^d}(\epsilon L)^d\eta(y)
\ge\frac{\gamma}{2}\biggr)
\le 2\exp\bigl\{-c\epsilon^{-d(1-\rho)+2a_0}\bigr\}
\end{equation}
for any $\gamma>0$. 
On the other hand, by definition of~$\zeta_j(x)$ we have 
\begin{equation}
\label{A.6'}
\BbbP\biggl(\sum_{y\in L\Z^d}(\epsilon L)^d\zeta_j(y)
\ge\frac{\gamma}{2J}\biggr)
\le \BbbP\biggl(\,\sum_{y\in L\Z^d}1_{\{\zeta_j(y)\neq 0\}}
\ge\frac{\gamma}{2J}\epsilon^{-d(1-\rho)+a_j}\biggr).
\end{equation}
Noting that $-d(1-\rho)+a_J<0$ and 
that $\{1_{\{\zeta_j(y)\neq 0\}}\}_{y\in L\Z^d}$ are stochastically dominated  
by independent Bernoulli variables with success probability bounded by 
\begin{equation}
\BbbP(\zeta_j(y)\ne0)\le\BbbP\bigl(\Xi_L(y)>\epsilon^{-a_{j-1}}\bigr)
\le \epsilon^{a_{j-1}K/r}\sup_{\epsilon\in(0,1)}
\sup_{y\in L\Z^d}\E\left[\Xi_L(y)^{K/r}\right]
\end{equation}
an application of the Bernstein 
inequality along with $a_{j-1}K/r>a_j$ again bounds the right-hand side of~\eqref{A.6'} by $\exp\{-\epsilon^{0-}\}$ for sufficiently small 
$\epsilon$.
\end{proofsect}

The key estimate in this section is again encapsulated into:

\begin{proposition}
\label{lemma3.5}
For all~$k\in\N$ there is~$c>0$ such that for all sufficiently small $\gamma>0$ and all sufficiently small~$\epsilon>0$,
\begin{equation}
E_{k,\epsilon,\gamma}\cap F_{\epsilon,\gamma}\subseteq
\bigl\{ \Lambda_k^{\epsilon}(\xi)\ge\Lambda_k-ck\gamma \bigr\} \,.
\end{equation}
In particular, under Assumption~\ref{ass1}, for any~$\delta>0$,
\begin{equation}
\BbbP\bigl(
\Lambda_k^{\epsilon}(\xi)\le\Lambda_k-\delta\bigr)
\,\underset{\epsilon\downarrow0}\longrightarrow\,0.
\end{equation}
\end{proposition}

In light of our general strategy of playing the variational problems \twoeqref{E:3.22}{E:3.22b} against each other, the proof starts with a conversion of discrete eigenfunctions to functions over~$\R^d$. This following lemma will be quite useful in this vain:

\begin{lemma}
\label{fe}
There is a constant~$C=C(d)$ for which the following holds: For
any function $f\colon\Z^d\to\R$ and any~$\epsilon\in(0,1)$, there is a
function $\wt f\colon \R^d\to\R$ such that
\begin{enumerate}
\item the map $f\mapsto\wt f$ is linear,
\item $\wt f$ is continuous on~$\R^d$ and $\wt f(x\epsilon)=f(x)$ for all $x\in \Z^d$,
\item
for any $x\in\Z^d$ and any $y\in \epsilon x+[0,\epsilon)^d$ we have
\begin{equation}
\label{E:3.13aa}
\bigl|\wt f(y)\bigr|\le\max_{z\in x+\{0,1\}^d}\,\bigl|f(z)\bigr|,
\end{equation}
and
\begin{equation}
\label{E:3.14uu}
\bigl|\wt f(y)-f(x)\bigr|\le d\max_{z\in x+\{0,1\}^d}\bigl|\nabla^{\ssup{\textd}}f(z)\bigr|,
\end{equation}
\item
for all~$p\in[1,\infty]$ we have
\begin{equation}
\label{E:3.15uu}
\Vert\wt f\Vert_{L^p(\R^d)}\le C(d)\Vert f\Vert_{\epsilon,p},
\end{equation}
and
\begin{equation}
\label{E:3.16q}
\sum_{x\in\Z^d}\int_{\epsilon x+[0,\epsilon)^d}
|\wt f(y)-f(x)|^2\textd y 
\le C(d)\|\nabla^{\ssup{\textd}} 
f\|_{\epsilon,2}^2,
\end{equation}
\item
$\wt f$ is piece-wise linear and thus almost everywhere 
differentiable with
\begin{equation}
\label{E:3.15q}
\|\nabla\wt f\|_{L^2(\R^d)}
=\epsilon^{-1}\|\nabla^{\ssup{\textd}}
f\|_{\epsilon,2}.
\end{equation}
\end{enumerate}
\end{lemma}
\begin{proof}
This is a restatement of Lemma~3.3 of~\cite{BFK14} (with a history of similar statements described there). 
\end{proof}

With this in hand, we are ready to give:

\begin{proofsect}{Proof of Proposition~\ref{lemma3.5}}
The proof will be based on Corollary~\ref{lemma-regularity} derived along with Proposition~\ref{prop2} whose $k\ge2$-part is in turn proved using the $k=1$-part of the statement under consideration. This poses no logical conflict since (as described in Remark~\ref{remark-excuse}), we first use Corollary~\ref{lemma-regularity} for~$k=1$, where no reference to the present statement is required, in the argument below to establish the present statement for~$k=1$. This then validates the proof of Proposition~\ref{prop2} and Corollary~\ref{lemma-regularity} for~$k\ge2$ which subsequently validates also the $k\ge2$-version of the proof below.

Let $g_{D_\epsilon,\xi}^{\ssup 1},\dots,g_{D_\epsilon,\xi}^{\ssup k}$ be
 (a choice of) an ONS of the first $k$ eigenfunctions 
of~$H_{D_\epsilon,\xi}$ 
and let $\wt g_{1,\xi}^{\,\epsilon},\dots,\wt g_{k,\xi}^{\,\epsilon}$ be functions on~$\R^d$ associated with
$\epsilon^{-d/2}g_{D_\epsilon,\xi}^{\ssup 1},\dots,
\epsilon^{-d/2}g_{D_\epsilon,\xi}^{\ssup k}$, 
respectively, as described in Lemma~\ref{fe}. 
Corollary~\ref{lemma-regularity} ensures
\begin{equation}
\label{E:3.47}
\sup_{0<\epsilon<1} \sup_{\xi\in E_{k,\epsilon,\gamma}\cap F_{\epsilon,\gamma}}
\epsilon^{-2}\Vert 
\nabla^{\ssup{\textd}}g_{D_\epsilon,\xi}^{\ssup i}\Vert_{2}^2<\infty
\end{equation}
and so, in light of parts (1) and~(4) of Lemma~\ref{fe},
\begin{equation}
\sup_{\xi\in E_{k,\epsilon,\gamma}\cap F_{\epsilon,\gamma}}
\Bigl|\langle\wt g_{i,\xi}^{\,\epsilon},\wt g_{j,\xi}^{\,\epsilon}
\rangle_{L^2(\R^d)}-\delta_{ij}\Bigr|\,\underset{\epsilon\downarrow 0}
\longrightarrow\,0.
\end{equation}
Invoking again the Gram-Schmidt orthogonalization, we can thus find 
functions $\wt h_{1,\xi}^\epsilon,\dots,\wt h_{k,\xi}^\epsilon$ and 
coefficients $a_{ij}(\xi,\epsilon)$, $1\le i,j\le k$, such that
\begin{equation}
\label{E:3.45}
\wt h_{i,\xi}^\epsilon=\sum_{j=1}^k\bigl(\delta_{ij}
+a_{ij}(\xi,\epsilon)\bigr)\wt g_{j,\xi}^{\,\epsilon},\qquad i=1,\dots,k,
\end{equation}
and
\begin{equation}
\label{E:3.46}
\bigl\langle \wt h_{i,\xi}^\epsilon,{\wt h}_{j,\xi}^\epsilon
\bigr\rangle_{L^2(\R^d)}=\delta_{ij}
\qquad\text{and}\qquad\max_{i,j}
\sup_{\xi\in E_{k,\epsilon,\gamma}\cap F_{\epsilon,\gamma}}\,\bigl|
a_{ij}(\xi,\epsilon)\bigr|\,\underset{\epsilon\downarrow 0}
\longrightarrow\,0.
\end{equation}
Thanks to the definition of~$D_\epsilon$, Lemma~\ref{fe}(3) and \eqref{E:3.45}, both  
$\wt g_{i,\xi}^{\,\epsilon}$ and $\wt h_{i,\xi}^{\epsilon}$ are supported in~$D$. 

Lemma~\ref{fe}(5) along with \eqref{E:3.47} and \twoeqref{E:3.45}{E:3.46} 
in turn guarantee
\begin{equation}
\sup_{\xi\in E_{k,\epsilon,\gamma}\cap F_{\epsilon,\gamma}}\,\Bigl|\,
\Vert \nabla \wt h_{i,\xi}^{\epsilon}\Vert_{L^2(\R^d)}^2
-\epsilon^{-2}\Vert \nabla^{\ssup{\textd}}g_{D_\epsilon,\xi}^{\ssup i}\Vert_{2}^2
\Bigr|\,\underset{\epsilon\downarrow0}\longrightarrow\,0
\end{equation}
while \eqref{E:3.16q} ensures
\begin{equation}
\sup_{\xi\in E_{k,\epsilon,\gamma}\cap F_{\epsilon,\gamma}}
\,\Bigl|\bigl\langle U,(\wt h_{i,\xi}^{\epsilon})^2\bigr\rangle_{L^2({\R^d})}
-\bigl\langle U(\epsilon\cdot),(g_{D_\epsilon,\xi}^{\ssup i})^2
\bigl\rangle\Bigr|\,\underset{\epsilon\downarrow0}\longrightarrow\,0.
\end{equation}
Using $\wt h_{i,\xi}^{\epsilon}$ as the $\psi_i$'s in \eqref{E:3.22b} 
and noting that the $g_{D_\epsilon,\xi}^{\ssup i}$'s achieve the 
infimum in \eqref{E:3.22}, we find
\begin{equation}
\Lambda_k\le\Lambda_k^\epsilon(\xi)+\gamma+\sum_{i=1}^k\bigl\langle 
U(\epsilon\cdot)-\xi,(g_{D_\epsilon,\xi}^{\ssup i})^2\bigr\rangle
\end{equation}
when $\epsilon$ is sufficiently small. 
Now we apply the piece-wise constant approximation defined 
in~\eqref{local-average} to the function 
$g_{D_\epsilon,\xi}^{\ssup i}$ and invoke H\"older's inequality with conjugate exponents~$(r,r')$, where~$r$ is as in \eqref{E:r-eq}, to obtain 
\begin{equation}
\label{ptl-energy}
\begin{split}
\bigl\langle U(\epsilon\cdot)-\xi,(g_{D_\epsilon,\xi}^{\ssup i})^2
\bigr\rangle
&\le \bigl\langle U(\epsilon\cdot)-\xi,
((g_{D_\epsilon,\xi}^{\ssup i})_L)^2\bigr\rangle\\
&\quad
+\epsilon^{-d/r}\Vert U(\epsilon\cdot)-\xi \Vert_{\epsilon,r}
\,\bigl\Vert (g_{D_\epsilon,\xi}^{\ssup i})^2
-((g_{D_\epsilon,\xi}^{\ssup i})_L)^2\bigr\Vert_{r'}.
\end{split}
\end{equation}
Using Lemma~\ref{lemma-3.4new}, Corollary~\ref{lemma-regularity} 
and Proposition~\ref{prop2}, we find 
\begin{equation}
\begin{split}
\bigl\Vert (g_{D_\epsilon,\xi}^{\ssup i})^2
 -((g_{D_\epsilon,\xi}^{\ssup i})_L)^2\bigr\Vert_{r'}
& \le cL\epsilon 
 \Vert g_{D_\epsilon,\xi}^{\ssup i}\Vert_{\frac{2r'}{2-r'}}\\
&{= cL\epsilon^{1+d/r}
 \Vert \epsilon^{-d/2}
 g_{D_\epsilon,\xi}^{\ssup i}\Vert_{\epsilon,\frac{2r'}{2-r'}}}\\
&\le c L\,\epsilon^{1+d/r}.
\end{split}
\end{equation}
Since $L=o(\epsilon^{-1})$, the second term on the right-hand side 
of~\eqref{ptl-energy} is negligible. On the event
$E_{k,\epsilon,\gamma}\cap F_{\epsilon,\gamma}$, the first term on 
the right-hand side of~\eqref{ptl-energy} is also bounded as
\begin{equation}
 \bigl|\bigl\langle U(\epsilon\cdot)-\xi,((g_{D_\epsilon,\xi}^{\ssup i})_L)^2
 \bigr\rangle\bigr|
\le \|\overline{\xi}_L\|_{\epsilon,r}
\|\epsilon^{-d/2}(g_{D_\epsilon,\xi}^{\ssup i})_L\|_{\epsilon,2r'}^2\le c\gamma,
\end{equation}
again by Lemma~\ref{lemma-3.4new}. We thus get
$\Lambda_k^\epsilon(\xi)\ge \Lambda_k-c\gamma$ on
$E_{k,\epsilon,\gamma} \cap F_{\epsilon,\gamma}$ for sufficiently 
small $\epsilon$, as desired. 
\end{proofsect}

\begin{proofsect}{Proof of Theorem~\ref{thm1.1}}
By Propositions~\ref{lemma-3.4} and~\ref{lemma3.5}, 
for any $\delta>0$ and $k\in\N$ we have
\begin{equation}
\BbbP\bigl(\,|\Lambda_k^\epsilon(\xi)-\Lambda_k|>\delta\bigr)
\,\underset{\epsilon\downarrow 0}
\longrightarrow\,0.
\end{equation}
Since
\begin{equation}
\lambda_{D_\epsilon,\xi}^{\ssup k}=
\Lambda_k^\epsilon(\xi)-\Lambda_{k-1}^\epsilon(\xi)
\quad\textrm{and}\quad\lambda_D^{\ssup k}=\Lambda_k-\Lambda_{k-1},
\end{equation}
the convergence of the individual eigenvalue follows. 
\end{proofsect}

The proof of Proposition~\ref{lemma3.5} gives us the following 
additional fact:

\begin{corollary}
\label{cor3.8}
Given any choice of $\xi\mapsto g_{D_\epsilon,\xi}^{\ssup 1}, \dots,
 g_{D_\epsilon,\xi}^{\ssup k}$, let $\wt
 g_{1,\xi}^{\,\epsilon},\dots,\wt g_{k,\xi}^{\,\epsilon}$ denote the
 continuum interpolations of 
$\epsilon^{-d/2}g_{D_\epsilon,\xi}^{\ssup 1},\dots,\epsilon^{-d/2}g_{D_\epsilon,\xi}^{\ssup k}$ 
as constructed in
Lemma~\ref{fe}. Assume $\lambda_D^{\ssup{k+1}}>\lambda_D^{\ssup k}$
and let~$\hat\Pi_k$ denote the orthogonal projection on
$\{\varphi_D^{\ssup1},\dots,\varphi_D^{\ssup k}\}^\perp$. 
Then, for any $\delta>0$, whenever $\gamma>0$ and $\epsilon>0$ are
sufficiently small, 
\begin{equation}
\label{E:3.60}
\biggl\{\,\xi\colon\sum_{i=1}^k\Vert\hat\Pi_k\wt 
g_{i,\xi}^{\,\epsilon}\Vert_{L^2({\R^d})}>{\delta}\biggr\}
\subseteq (E_{k,\epsilon,\gamma}\cap F_{\epsilon,\gamma})^\cc.
\end{equation}
\end{corollary}

\begin{proofsect}{Proof}
This is proved in the same way as Corollary~3.8 of~\cite{BFK14}. 
\end{proofsect}

We close this subsection with an $\ell^{\infty}$-bound for the eigenfunction. Compared with the case of bounded $\xi$ (cf.~Lemma~3.2 of~\cite{BFK14}), the bound is weaker but it is still useful in the proof of Theorem~\ref{thm1.2}.

\begin{lemma}
\label{prop4}
For~all $p>1$, all~$k\in\N$ and all sufficiently small~$\gamma>0$ there is $c_{k,p,\gamma}$ such that for all $\epsilon\in(0,1)$,
\begin{equation}
\label{E:3.41}
E_{k,\epsilon,\gamma}\cap F_{\epsilon,\gamma}
\subseteq\bigl\{\Vert g_{D_\epsilon,\xi}^{\ssup k}
\Vert_{\infty}^2\le c_{k,p,\gamma}\,{\epsilon^{d/p}}\bigr\}.
\end{equation}
\end{lemma}

\begin{proofsect}{Proof}
Let $\{X_t\colon t\ge0\}$  denote the (constant speed) continuous-time 
simple symmetric random walk on~$\Z^d$ killed upon exiting from~$D_\epsilon$. 
The eigenvalue equation and the Feynman-Kac formula imply
\begin{equation}
\begin{aligned}
g_{D_\epsilon,\xi}^{\ssup k}(x)&=\texte^{t\lambda_{D_\epsilon,\xi}^{\ssup k}}
\bigl(
\texte^{-tH_{D_\epsilon,\xi}}g_{D_\epsilon,\xi}^{\ssup k}\bigr)(x)\\
&=\texte^{t\lambda_{D_\epsilon,\xi}^{\ssup k} } 
E^x\biggl(\exp\Bigl\{-\int_0^{t\epsilon^{-2}}
\epsilon^2\xi(X_s)\textd s\Bigr\}
g_{D_\epsilon,\xi}^{\ssup k}(X_{t\epsilon^{-2}})\biggr),
\end{aligned}
\end{equation}
where $E^x$ denotes the expectation over the walk started at~$x$. 
Writing $p_t(x,y)$ for the probability that the walk started at~$x$ is at~$y$ at time~$t$, H\"older's inequality with conjugate indices~$(p,q)$ yields
\begin{equation}
\begin{split}
  \bigl|g_{D_\epsilon,\xi}^{\ssup k}(x)\bigr|
 &\le\texte^{t\lambda_{D_\epsilon,\xi}^{\ssup k}} 
E^x\biggl(\exp\Bigl\{-\int_0^{t\epsilon^{-2}}
 q\epsilon^2\xi(X_s)\textd s\Bigr\}\biggr)^{1/q}
 E^x\Bigl(\bigl|g_{D_\epsilon,\xi}^{\ssup k}(X_{t\epsilon^{-2}})\bigr|^p\Bigr)^{1/p}\\
 &\le\texte^{t\lambda_{D_\epsilon,\xi}^{\ssup k}}
 \Bigl\langle\delta_x, 
 \texte^{-t H_{D_\epsilon,q\xi}}1\Bigr\rangle^{1/q}
 \Biggl(\sum_{y\in D_\epsilon}p_{t\epsilon^{-2}}(x,y)
 \bigl|g_{D_\epsilon,\xi}^{\ssup k}(y)\bigr|^p\Biggr)^{1/p}.
\end{split}
\end{equation}
The ($1/q$-th power of the) inner product on the right-hand side is bounded by
\begin{equation}
\begin{split}
\left( \|\delta_x\|_2 \,\bigl\|\texte^{-t H_{D_\epsilon,q\xi}}
\bigr\|_{\ell^2\to \ell^2}\,\|1\|_2\right)^{1/q}
\le
ce^{-t\lambda^{\ssup 1}_{D_\epsilon, q\xi}/q}
\epsilon^{-d/2q}. 
\end{split}
\end{equation}
On the other hand, invoking the Cauchy-Schwarz inequality and 
using Proposition~\ref{prop2} we get
\begin{equation}
\label{E:3.10aa}
\begin{split}
 \Biggl(\sum_{y\in D_\epsilon}p_{t\epsilon^{-2}}(x,y)
 \bigl|g_{D_\epsilon,\xi}^{\ssup k}(y)\bigr|^p \Biggr)^2
 &\le
 \sum_{y\in D_\epsilon}p_{t\epsilon^{-2}}(x,y)^2
 \sum_{y\in D_\epsilon} \vert g_{D_\epsilon,\xi}^{\ssup k}(x)\vert^{2p}\\
 &\le  cp_{2t\epsilon^{-2}}(x,x)
\epsilon^{d(p-1)}, \quad\text{on }E_{k,\epsilon,\gamma},
\end{split}
\end{equation}
where in the second inequality we have used the fact that 
$p_{t\epsilon^{-2}}(\cdot,\cdot)$ is symmetric.
Since $p_{\cdot}$ is bounded by the transition kernel
of the random walk without killing, the local central limit theorem 
yields $p_{2t\epsilon^{-2}}(x,x)\le ct^{-d/2}\epsilon^d$. 
Summarizing the above bounds, we arrive at
\begin{equation}
 |g_{D_\epsilon,\xi}^{\ssup k}(x)|^2
\le c\exp\bigl\{t\lambda^{\ssup k}_{D_\epsilon, \xi}
-t\lambda^{\ssup 1}_{D_\epsilon, q\xi}/q\bigr\}
{t^{-d/2p}}\epsilon^{d(1-1/q)}.
\end{equation}
The desired bound follows by taking~$t:=1$ and noting that, by Corollary~\ref{cor-3.4} and Proposition~\ref{lemma3.5}, the eigenvalues are bounded on~$E_{k,\epsilon,\gamma}\cap F_{\epsilon,\gamma}$ uniformly in~$\epsilon$. 
\end{proofsect}

\begin{remark}
For $d=1$, the bound \eqref{E:3.41} holds (with a finite constant) even for~$p=1$. This follows from Corollary~\ref{lemma-regularity} and a discrete version
of Morrey's inequality.
\end{remark}

\section{Gaussian limit law}
\label{sec5}\noindent
We are now finally ready to address the second main aspect of this work, 
which is the limit theorem for fluctuations of asymptotically 
non-degenerate eigenvalues.
Just as Lemma~\ref{lemma-2.1}, we have the following fact that allows us to work with a truncated potential.
\begin{lemma}
\label{lemma-3.1}
Under Assumption~\ref{ass2}, for each $\kappa\in (d/K, 2 \wedge d/2)$ we have
\begin{equation}
\BbbP\Bigl(\,
\max_{x\in D_\epsilon}|\xi(x)|> \epsilon^{-\kappa}\Bigr)\,\underset{\epsilon\downarrow0}\longrightarrow\,0.
\end{equation}
\end{lemma}
We fix $\kappa\in (d/K, 2 \wedge d/2)$ and assume
\begin{equation}
\label{xi-bounded2}
 \max_{x\in D_\epsilon}|\xi(x)|\le \epsilon^{-\kappa}
\end{equation}
in what follows.

As in our earlier work~\cite{BFK14} (and drawing inspiration from~\cite{BSW}),  the main idea is to use a martingale central limit theorem. 
Consider an ordering of the vertices in~$D_\epsilon$ into a sequence 
$x_1,\dots,x_{|D_\epsilon|}$ 
and let $\FF_m:=\sigma(\xi({x_1}),\dots,\xi({x_m}))$. Then
\begin{equation}
 \lambda_{D_\epsilon,\xi}^{\ssup{k}}
 -\E \lambda_{D_\epsilon,\xi}^{\ssup{k}}
=\sum_{m=1}^{|D_\epsilon|}Z_m^{\ssup{k}}\,,\quad\text{where}\quad
Z_m^{\ssup{k}}:=\E\bigl(\lambda_{D_\epsilon,\xi}^{\ssup{k}}\big|\FF_m\bigr)
-\E\bigl(\lambda_{D_\epsilon,\xi}^{\ssup{k}}\big|\FF_{m-1}\bigr),
\end{equation}
represents the fluctuation of the $k$-th eigenvalue as a martingale. 
We shall appeal to the Martingale Central Limit Theorem due to 
Brown~\cite{Brown} which yields Theorem~\ref{thm1.2} under the following conditions:
\begin{enumerate}
\item[(1)] 
if $\lambda_D^{\ssup{i}}$ and $\lambda_D^{\ssup{j}}$ are simple, then 
\begin{equation}
\epsilon^{-d}\sum_{m=1}^{|D_\epsilon|}
\E\bigl(Z_m^{\ssup{i}}Z_m^{\ssup{j}}\big|\FF_{m-1}\bigr)
\,\,\underset{\epsilon\downarrow0}{\overset{\BbbP}\longrightarrow}\,\,
\sigma_{ij}^2,
\end{equation}
\item[(2)] for each $\delta>0$ and each $i\ge 1$,
\begin{equation}
\epsilon^{-d}\sum_{m=1}^{|D_\epsilon|}
\E\bigl((Z_m^{\ssup{i}})^2\1_{\{|Z_m^{\ssup{i}}|>
\delta\epsilon^{d/2}\}}\big|\FF_{m-1}\bigr)
\,\,\underset{\epsilon\downarrow0}{\overset{\BbbP}\longrightarrow}\,\,0.
\end{equation}
\end{enumerate}
In order to control the limits in (1) and (2), we rewrite the martingale 
difference by using an independent copy $\widehat\xi$ of~$\xi$ as 
\begin{equation}
Z_m^{\ssup{i}}= \widehat\E\Bigl(
\lambda_{D_\epsilon,\widehat\xi^{(m)}}^{{\ssup{i}}}
-\lambda_{D_\epsilon,\widehat\xi^{(m-1)}}^{{\ssup{i}}}\Bigr),
\end{equation}
where $\widehat\E$ is the expectation corresponding to $\widehat\xi$
and $\widehat\xi^{(m)}$ denotes the configuration
\begin{equation}
\widehat\xi^{(m)}(x_i):=
\begin{cases}
\xi(x_i),\qquad&\text{if }i\le m,\\
\widehat\xi(x_i),\qquad&\text{if }i> m.
\end{cases}
\end{equation}
\begin{lemma}
\label{lemma-5.1}
The function $\xi\mapsto\lambda_{D_\epsilon,\xi}^{\ssup{k}}$ is 
everywhere right and left differentiable with respect to each~$\xi(x)$. 
For each~$\xi$, the set of values of~$\xi(x)$ where the right and left partial derivatives with respect to~$\xi(x)$ disagree is finite; 
else the derivative exists and is continuous in~$\xi(x)$. 
At the point of differentiability, the partial derivative 
$\frac\partial{\partial\xi(x)}\lambda_{D_\epsilon,\xi}^{\ssup{k}}$ 
obeys 
\begin{equation}
\label{E:2.3}
\frac\partial{\partial\xi(x)}\lambda_{D_\epsilon,\xi}^{\ssup{k}}
=g_{D_\epsilon,\xi}^{\ssup k}(x)^2
\end{equation}
for any possible choice of $g_{D_\epsilon,\xi}^{\ssup k}$. 
(I.e., all choices give the same result.) 
\end{lemma}

\begin{proofsect}{Proof}
This is a classical result in the matrix analysis called Hadamard's first 
variation formula. In the analytic perturbation theory of self-adjoint 
operators, it is also called Feynman-Herman formula. 
See, for example, Reed and Simon~\cite{RS4}, Theorem XII.3 and the computation
of the Rayleigh--Schr\"odinger coefficients presented on pages~5--8 thereof.
An elementary proof of a slightly weaker assertion can be found 
in~\cite{BFK14}.
\end{proofsect}

This lemma allows us to further rewrite the martingale difference, 
by using the fundamental theorem of calculus, as 
\begin{equation}
\label{E:2.2}
 \widehat\E\Bigl(
\lambda_{D_\epsilon,\widehat\xi^{(m)}}^{{\ssup{i}}}
-\lambda_{D_\epsilon,\widehat\xi^{(m-1)}}^{{\ssup{i}}}\Bigr)
\phantom{:}=\widehat\E\biggl(\int_{\widehat\xi(x_m)}^{\xi(x_m)}
g_{D_\epsilon,\wt\xi^{(m)}}^{\ssup i}({x_m})^2\textd\tilde\xi\biggr),
\end{equation}
where $\wt\xi^{(m)}$ is the configuration that equals $\xi$ on 
$\{x_1,\dots,x_{m-1}\}$, 
coincides with~$\widehat\xi$ on $\{x_{m+1},\dots,x_{|D_\epsilon|}\}$ and takes value $\tilde\xi$ at~$x_m$. 
The integral is to be understood in the Riemann sense, meaning in particular that the sign 
changes upon exchanging the limits of integration. 

For condition~(1), we will proceed by replacing the square of the discrete 
eigenfunction by its corresponding continuum counterpart. 
As in~\cite{BFK14}, the main task is to get
rid of the dummy variable $\tilde\xi$ by showing that 
changing the value of $\xi$ at one point causes little effect on the
eigenfunction. 

\begin{lemma}
\label{prop3}
Given~$k\in\N$ and a configuration $\xi$, suppose that 
$\lambda_{D_\epsilon,\xi}^{\ssup k}$ remains simple as~$\xi(x)$ 
varies in $[-\epsilon^{-\kappa},\epsilon^{-\kappa}]$. 
Then for any $\xi'$ satisfying $\xi(y)=\xi'(y)$ for~$y\ne x$ and for 
any $\xi(x)$ and $\xi'(x)$,
\begin{equation}
\label{E:g-g}
\bigl|g_{D_{\epsilon},\xi'}^{\ssup k}(x)\bigr| 
= \bigl|g_{D_{\epsilon},\xi}^{\ssup k}(x)\bigr|
\exp\biggl\{\int_{\xi(x)}^{\xi'(x)}G_{D_\epsilon}^{\ssup k}
(x,x;\tilde\xi)\,\textd\tilde\xi(x)\biggr\},
\end{equation}
where~$\tilde\xi$ is the configuration that agrees with $\xi$ 
(and~$\xi'$) outside~$x$ where it equals~$\tilde\xi(x)$ and
\begin{equation}
G_{D_\epsilon}^{\ssup k}(x,y;\xi):=\bigl\langle\delta_x,
(H_{D_\epsilon,\xi}-\lambda_{D_\epsilon,\xi}^{\ssup k})^{-1}
(1-\widehat P_k)\delta_y\bigr\rangle_{\ell^2(\Z^d)}
\end{equation}
with~$\widehat P_k$ denoting the orthogonal projection on 
$\text{\rm Ker}(\lambda_{D_\epsilon,\xi}^{\ssup k}-H_{D_\epsilon,\xi})$.
\end{lemma}

\begin{proofsect}{Proof}
This follows from the so-called Hadamard's second variation formula. 
See Lemma~5.2 of~\cite{BFK14} for a direct proof.
\end{proofsect}

Our next lemma shows that when $\lambda_D^{\ssup k}$ is simple, 
the random eigenvalue $\lambda_{D_\epsilon,\xi}^{\ssup k}$ indeed remains 
simple as~$\xi(x)$ varies in $[-\epsilon^{-\kappa},\epsilon^{-\kappa}]$ and also the term in the exponent of \eqref{E:g-g} 
tends to zero as~$\epsilon\downarrow0$ with very high probability. 
Let us fix $p>1$ such that 
\begin{equation}
\label{p}
d/p-\kappa>d/2\qquad\textrm{and}\qquad d/p-\kappa+2\wedge d>d,
\end{equation}
recalling \eqref{xi-bounded2}.  
Further, we set
\begin{equation}
 \delta:=\frac13\min\{\lambda_D^{\ssup k}-\lambda_D^{\ssup{k-1}},
\lambda_D^{\ssup{k+1}}-\lambda_D^{\ssup k}\}
\end{equation}
and define the events
\begin{align}
A^1_{k,\epsilon}&:=\bigcap_{x\in
D_\epsilon}\Bigl\{\xi\colon\sup_{\xi(x)}
|\lambda_{D_\epsilon,\xi}^{\ssup i}-\lambda_D^{\ssup i}|
<\delta \textrm{ for all }
1\le i\le k+1\Bigr\},\\
 A^2_{k,\epsilon}&:=\bigcap_{x\in D_\epsilon}
\Bigl\{\xi\colon\sup_{\xi(x)}
 \Bigl|G^{\ssup k}_{D_\epsilon}(x,x;\xi)\Bigr|
 \le G(\epsilon)\Bigr\}
\end{align}
with the suprema over~$\xi(x)$ over $[-\epsilon^{-\kappa},\epsilon^{-\kappa}]$ and 
\begin{equation}
\label{unif-g2}
G(\epsilon):=c_G\times
\begin{cases}
\epsilon,& d=1,\\
\epsilon^2\log\frac{1}{\epsilon},& d=2,\\
\epsilon^2,& d\ge 3.
\end{cases}
\end{equation}
where~$c_G$ is to be determined momentarily. 
Abbreviate 
\begin{equation}
A_{k,\epsilon,\gamma}:= 
A^1_{k,\epsilon}\cap 
A^2_{k,\epsilon}\cap E_{k,\epsilon,\gamma}
\cap F_{\epsilon,\gamma}\,.
\end{equation}
We then have:

\begin{lemma}
\label{unif-g}
If $\lambda_D^{\ssup k}$ is simple and $c_G$ in~\eqref{unif-g2}
is chosen sufficiently large, then for all~$\gamma>0$ and~$\epsilon>0$ sufficiently small, 
\begin{equation}
\BbbP(A_{k,\epsilon,\gamma})\ge 1-\exp\{-\epsilon^{0-}\}.
\end{equation}
\end{lemma}

\begin{proofsect}{Proof}
It readily follows from Propositions~\ref{lemma-3.4} 
and~\ref{lemma3.5} that, for some constant~$c>0$, 
\begin{equation}
\sup_{\xi\in E_{k,\epsilon,\gamma}\cap F_{\epsilon,\gamma}}
\max_{1\le i\le k+1}
|\lambda_{D_\epsilon,\xi}^{\ssup i}-\lambda_D^{\ssup i}|
<c\gamma
\end{equation}
holds for sufficiently small $\gamma>0$ and $\epsilon>0$.  
Now for any~$\eta$ which differs from 
$\xi\in E_{k,\epsilon,\gamma}\cap F_{\epsilon,\gamma}$ only at $x$, 
one can easily check that 
$\eta\in E_{k,\epsilon,2\gamma}\cap F_{\epsilon,2\gamma}$ up to a change of the constant explained in Remark~\ref{E_k-change}.
For instance, if $\|\xi\|_{\epsilon,r}<4|D|\max_{x\in D_\epsilon}\E[|\xi(x)|^r]$, then for small enough~$\epsilon>0$, 
\begin{equation}
\sup_{\xi(x)}\|\xi\|_{\epsilon,r}\le \|\xi\|_{\epsilon,r}+\epsilon^{d/r-\kappa}
\le
5|D|\max_{x\in D_\epsilon}\E[|\xi(x)|^r]
\end{equation}
follows from our choice $r<d/\kappa$. Therefore, by Lemmas~\ref{good-event} and~\ref{good-event2}, for each $x\in D_\epsilon$ and with the supremum over~$\xi(x)$ restricted to~$[-\epsilon^{-\kappa},\epsilon^{-\kappa}]$,
\begin{equation}
 \BbbP\biggl(\,\sup_{\xi(x)}
|\lambda_{D_\epsilon,\xi}^{\ssup i}-\lambda_D^{\ssup i}|
<\delta \textrm{ for all }
1\le i\le k+1\biggr)
\ge 1-\exp\{-\epsilon^{0-}\}.
\end{equation}
Since $|D_\epsilon|=O(\epsilon^{-d})$, the union bound yields
\begin{equation}
 \BbbP(A_{k,\epsilon}^1)\ge 1-\exp\{-\epsilon^{0-}\}
\end{equation}
for all~$\gamma>0$ and~$\epsilon>0$ sufficiently small. 

Next, we estimate the probability of $A_{k,\epsilon}^2$. 
Hereafter, we assume that $\xi\in A_{k,\epsilon}^1$. 
Then $\min_{i\in \N\setminus\{k\}}{|\lambda_{D_\epsilon,\xi}^{\ssup i}-
 \lambda_{D_\epsilon,\xi}^{\ssup k}|}$ is at least $\delta$ and if we choose $\lambda$ so that $\lambda+\lambda_D^{\ssup 1}>\delta$, then for some constant~$c>0$ depending only on~$\lambda$ and~$k$, 
\begin{equation}
\label{E:3.21}
\bigl|G_{D_\epsilon}^{\ssup k}(x,x;\xi)\bigr|
\le\sum_{\begin{subarray}{c}
 i\ge 1\\ i\ne k
 \end{subarray}}
 \,\frac1{|\lambda_{D_\epsilon,\xi}^{\ssup i}-
 \lambda_{D_\epsilon,\xi}^{\ssup k}|}
 g_{D_\epsilon,\xi}^{\ssup i}(x)^2
 \le c\sum_{i\ge 1}
 \,\frac{1}{\lambda+\lambda_{D_\epsilon,\xi}^{\ssup i}}
 g_{D_\epsilon,\xi}^{\ssup i}(x)^2,\quad \xi\in A_{k,\epsilon}^1.
\end{equation}
The sum on the right-hand side is nothing but the $\lambda$-Green 
kernel of $H_{D_\epsilon,\xi}$ evaluated at $(x,x)$.
Let us define
\begin{equation}
I_{t,z}(\xi)
:=E^z\biggl[\int_0^{t\epsilon^{-2}}\epsilon^2|\xi|(X_s)\textd s\biggr]
=\epsilon^2\int_0^{t\epsilon^{-2}}\sum_{y \in D_\epsilon}p_t(z,y)
|\xi|(y)\textd s,
\end{equation}
where $p$ and $X$ are the same as in the proof of Lemma~\ref{prop4}. 
Using the Cauchy-Schwarz inequality and a standard heat kernel bound, 
we obtain
\begin{equation}
\label{Lip-I}
\begin{split}
\bigl |I_{t,z}(\xi)-I_{t,z}(\eta)\bigr|
 &\le \epsilon^2
\int_0^{t\epsilon^{-2}}
\biggl(\sum_{y \in D_\epsilon}p_t(z,y)^2\biggr)^{1/2}
\biggl(\sum_{y \in D_\epsilon}|\xi(y)-\eta(y)|^2\biggr)^{1/2}
\textd s\\
&=\epsilon^2
\Bigl(\,\int_0^{t\epsilon^{-2}}p_{2t}(z,z)^{1/2}\textd s\Bigr)\,\|\xi-\eta\|_2\\
&\le c\|\xi-\eta\|_2\times
\begin{cases}
 t^{1-d/4}\epsilon^{d/2},&d\le 3,\\
 \epsilon^2\log (t\epsilon^{-2}),&d=4,\\
 \epsilon^2,&d\ge 5.
\end{cases}
\end{split}
\end{equation}
Noting also that $I_{t,z}(\cdot)$ is linear and $|I_{t,z}|\le t\epsilon^{-\kappa}$ thanks to \eqref{xi-bounded2}, we may use Talagrand's concentration inequality (Theorem~6.6 of Talagrand~\cite{Tal96}) and~\eqref{xi-bounded2} to get 
\begin{equation}
\begin{split}
 \max_{z\in D_\epsilon}\,\BbbP\bigl(|I_{t,z}(\xi)-\text{med}(I_{t,z})|> R\bigr)
& \le 4\exp\bigl\{-cR^2\epsilon^{2\kappa-4\wedge d}/\log(\epsilon^{-1})\bigr\}\\
& \le \exp\{-cR^2\epsilon^{0-}\}
\end{split}
\label{Talagrand}
\end{equation}
for all $R>0$, where~$c$ is a constant depending only on~$t$
and the bound holds for all $\epsilon>0$ sufficiently small. 
By integrating this bound, we first find $|\E(I_{t,z}) - \text{med}(I_{t,z})|<1/16$ for~$\epsilon>0$ small. 
Then for $t=(16\max_{x\in D_\epsilon}\E(|\xi(x)|))^{-1}$, we have $|\E(I_{t,z})|\le 1/16$ and hence $|\text{med}(I_{t,z})|<1/8$. By using this in~\eqref{Talagrand} and choosing $R=1/8$, we obtain the bound
\begin{equation}
\begin{split}
 \max_{z\in D_\epsilon}\,\BbbP\biggl(I_{t,z}(\xi)> \frac{1}{4}\biggr)
 \le \exp\{-\epsilon^{0-}\}
\end{split}
\end{equation}
for all sufficiently small $\epsilon>0$. 
Since~\eqref{Lip-I} ensures that varying~$\xi(x)$ over $[-\epsilon^\kappa,\epsilon^\kappa]$ brings only
$o(1)$ change to $I_{t,z}(\xi)$ and since 
$|D_\epsilon|=O(\epsilon^{-d})$, the union bound yields
\begin{equation}
\BbbP\biggl(\bigcup_{x\in D_\epsilon}\biggl\{
\sup_{\xi(x)}\sup_{z\in D_\epsilon}I_{t,z}(\xi)>{\frac{1}{3}}\biggr\}
\biggr)
\le \exp\{-\epsilon^{0-}\}
\end{equation} 
for $\epsilon>0$ sufficiently small. 
Now if $\sup_{z\in D_\epsilon}|I_{t,z}(\xi)|\le {1/3}$, 
a standard argument using Khas'minskii's lemma 
(see, e.g., Proposition~3.1 in Chapter~1 of Sznitman~\cite{Sznitman}) tells us that
\begin{equation}
\texte^{-sH_{D_\epsilon,\xi}}(x,x)
\le \zeta^{-1}\texte^{\zeta s}p_{2s\epsilon^{-2}}(x,x)
\end{equation}
for some universal constant $\zeta>0$. 
Multiplying both sides of this inequality by $\texte^{-\lambda s}$ with 
$\lambda>2\zeta\vee(\delta-\lambda_{D_\epsilon,\xi}^{\ssup 1})$ and 
integrating over $s\in (0,\infty)$, we obtain
\begin{equation}
(\lambda-H_{D_\epsilon,\xi})^{-1}(x,x) 
\le \frac{c}{\zeta}(\zeta-\epsilon^{-2}\Deltad)^{-1}(x,x)
\asymp
\begin{cases}
\epsilon,& d=1,\\
\epsilon^2\log\frac{1}{\epsilon},& d=2,\\
\epsilon^2,& d\ge 3.
\end{cases}
\end{equation}
Using this in \eqref{E:3.21} then yields a corresponding bound on~$\BbbP(A_{k,\epsilon}^2)$. 
\end{proofsect}

Now we are in position to check the conditions of the
Martingale Central Limit Theorem. 
Let us first check the condition~(2). 

\begin{proposition}
\label{MartCLT2}
For each $\delta>0$ and $i\ge 1$,
\begin{equation}
 \epsilon^{-d}\sum_{m=1}^{|D_\epsilon|}
 \E\bigl((Z_m^{\ssup{i}})^2\1_{\{|Z_m^{\ssup{i}}|>
 \delta\epsilon^{d/2}\}}\big|\FF_{m-1}\bigr)
 \,\,\underset{\epsilon\downarrow0}{\overset{\BbbP}\longrightarrow}
 \,\,0.
\end{equation}
\end{proposition}

\begin{proofsect}{Proof}
On the event $A_{k,\epsilon,\gamma}$, 
by using Lemma~\ref{prop4} in~\eqref{E:2.2}, we have 
\begin{equation}
\label{E:5.21*}
\sup_{\xi\in A_{k,\epsilon,\gamma}}
|Z_m^{\ssup i}|\le c\epsilon^{d/p-\kappa}.
\end{equation}
Thanks to~\eqref{p}, the right-hand side is $o(\epsilon^{d/2})$. 
On the other hand, 
$\sup_\xi|Z_m^{\ssup{k}}|_\infty\le 2\epsilon^{-\kappa}$ due to the 
truncation. 
From these bounds and Lemma~\ref{unif-g}, we obtain 
\begin{equation}
 \epsilon^{-d}\sum_{m=1}^{|D_\epsilon|}
 \E\bigl((Z_m^{\ssup{k}})^2\1_{\{|Z_m^{\ssup{k}}|>
 \delta\epsilon^{d/2}\}}\bigr)
 \le \epsilon^{-d}\sum_{m=1}^{|D_\epsilon|}
 \E\bigl((Z_m^{\ssup{k}})^2\1_{A_{k,\epsilon,\gamma}^\cc}\bigr)
 \le \exp\{-\epsilon^{0-}\}
\end{equation}
for sufficiently small $\epsilon$. This shows that the desired
convergence holds in $L^1(\BbbP)$, and thus also in probability. 
\end{proofsect}

Next we address condition~(1) of the Martingale Central Limit Theorem:
 
\begin{proposition}
\label{prop-5.4}
Suppose~$\lambda_D^{{\ssup{i}}}$ and $\lambda_D^{{\ssup{j}}}$ are simple. 
Abbreviate $B_\epsilon(x):=\epsilon x+[0,\epsilon)^d$. Then
\begin{equation}
\label{E:5.22}
\E\,\Biggl|
\,\sum_{m=1}^{|D_\epsilon|}\biggl( \E\bigl(
 (\epsilon^{-d} Z_m^{\ssup i})(\epsilon^{-d} Z_m^{\ssup j})\,
\big|\,\FF_{m-1}\bigr)-\int_{B_\epsilon(x_m)}\!\!\textd y\,\,
V(y)\varphi_D^{{\ssup{i}}}(y)^2
\varphi_D^{{\ssup{j}}}(y)^2\biggr)\Biggr|\,
\underset{\epsilon\downarrow 0}\longrightarrow\,0.
\end{equation}
\end{proposition}

The proof of this proposition will be done in several steps. 
Recall the definition of event $A_{k,\epsilon,\gamma}$ and note that,
on~$A_{k,\epsilon,\gamma}$ the eigenfunction $g_{D_\epsilon,\xi}^{\ssup k}$ is
unique up to a sign and, in particular, there is a unique measurable
version of~$\xi\mapsto g_{D_\epsilon,\xi}^{\ssup k}(x)^2$ for
each~$x$. We first eliminate the dummy variable~$\tilde \xi$. 

\begin{lemma}
\label{lemma-5.5}
Suppose $\lambda_D^{\ssup k}$ is simple. Then 
\begin{equation}
\label{E:5.25}
\epsilon^{-d}\sum_{m=1}^{|D_\epsilon|}
\,
\E
\Biggl(\,\biggl| \, Z_m^{\ssup k}-\bigl(\xi(x_m)-U(\epsilon x_m)\bigr)
\E\Bigl(\,g_{D_\epsilon,\xi}^{\ssup {k}}(x_m)^2
\1_{A_{k,\epsilon,\gamma}}\,\Big|\,\FF_m\Bigr)\biggr|^2\Biggr)\,
\underset{\epsilon\downarrow0}\longrightarrow\,0.
\end{equation}
\end{lemma}

\begin{proofsect}{Proof}
Inserting the indicator of $ \{\widehat\xi^{(m)}\in A_{k,\epsilon,\gamma}\}$ 
and/or its complement into the right-hand side of \eqref{E:2.2} and 
using the obvious bound 
$\sup_\xi\|g_{D_\epsilon,\xi}^{\ssup k}\|_\infty\le 1$, we get
\begin{equation}
\label{E:5.25a}
\Biggl|\,Z_m^{\ssup k}-\widehat\E\biggl(\1_{\{\widehat\xi^{(m)}
\in A_{k,\epsilon,\gamma}\}}\int_{\widehat\xi(x_m)}^{\xi(x_m)}
g_{D_\epsilon,\wt\xi^{(m)}}^{\ssup k}({x_m})^2\textd
\tilde\xi\biggr)\Biggr|
\le 2\epsilon^{-\kappa}\,\E(\1_{A_{k,\epsilon,\gamma}^\cc}|\FF_{m}).
\end{equation}
Abbreviate temporarily
\begin{equation}
F_m(\tilde\xi^{\ssup m}):=\exp\biggl\{2\int_{\xi(x_m)}^{\tilde\xi}
G_{D_\epsilon}^{\ssup{k}}({x_m,x_m};\tilde\xi')\textd\tilde\xi'\biggr\}.
\end{equation}
On the event $\{\widehat\xi^{(m)}\in A_{k,\epsilon,\gamma}\}$, 
Lemmas~\ref{prop3} yields
\begin{equation}
\begin{aligned}
\Bigl(\,\int_{\widehat\xi(x_m)}^{\xi(x_m)}
g_{D_\epsilon,\wt\xi^{(m)}}^{\ssup k}(x_m)^2\textd\tilde\xi\Bigr)
&-\bigl({\xi(x_m)}-{\widehat\xi(x_m)}\bigr)
g_{D_\epsilon,\widehat\xi^{(m)}}^{\ssup k}(x_m)^2\\
&\quad=\int_{\widehat\xi(x_m)}^{\xi(x_m)}\Bigl(
g_{D_\epsilon,\wt\xi^{(m)}}^{\ssup k}(x_m)^2
-g_{D_\epsilon,\widehat\xi^{(m)}}^{\ssup k}(x_m)^2
\Bigr)\textd\tilde\xi\\
&\quad=g_{D_\epsilon,\widehat\xi^{(m)}}^{\ssup k}(x_m)^2\,
\int_{\widehat\xi(x_m)}^{\xi(x_m)}\bigl
(F_m(\tilde\xi^{\ssup m})-1\bigr)\textd\tilde\xi
\end{aligned}
\end{equation}
and the last integral is estimated by using Lemma~\ref{unif-g} as
\begin{equation}
\left|\int_{\widehat\xi(x_m)}^{\xi(x_m)}\bigl(F_m(\tilde\xi^{\ssup m})-1
\bigr)\textd\tilde\xi\right|
\le 4|\xi(x_m)-\widehat\xi(x_m)|\epsilon^{-\kappa}G(\epsilon).
\end{equation}
This and \eqref{E:5.25a}, together with Lemma~\ref{prop4}, yield
\begin{equation}
\begin{split}
 \Biggl|\,Z_m^{\ssup k}-\widehat\E\biggl(\1_{\{\widehat\xi^{(m)}\in 
 A_{k,\epsilon,\gamma}\}}\bigl({\xi(x_m)}&-{\widehat\xi(x_m)}\bigr)
 g_{D_\epsilon,\widehat\xi^{(m)}}^{\ssup k}(x_m)^2
\biggr)\Biggr|^2\\
 &\quad\le c\bigl(\epsilon^{-\kappa}
\E(\1_{A_{k,\epsilon,\gamma}^\cc}|\FF_{m})^2+
 \epsilon^{2d/p-2\kappa}G(\epsilon)^2
\widehat\E(|\xi(x_m)-\widehat\xi(x_m)|^2)\bigr).
\end{split}
\end{equation}
As the configuration $\widehat\xi^{(m)}$ does not depend 
on~$\widehat\xi(x_m)$, we may take expectation with respect 
to~$\widehat\xi(x_m)$ and effectively replace it by $U(\epsilon x)$.
Taking the expectation over $\xi$ and 
summing over $x\in D_\epsilon$, we find that
the left-hand side of~\eqref{E:5.25} is bounded by 
\begin{equation}
\epsilon^{-d}\sum_{m=1}^{|D_\epsilon|}
 c\left(\epsilon^{-\kappa}
\BbbP(A_{k,\epsilon,\gamma}^\cc)+\epsilon^{2(d/p-\kappa+2\wedge d)}
 \log\frac{1}{\epsilon}
 \right)
 \le \epsilon^{-2d-2\kappa}\exp\{-\epsilon^{0-}\}
+\epsilon^{0+}
\end{equation}
by Lemma~\ref{unif-g} and~\eqref{p}.
\end{proofsect}

Next we bound the difference between the continuum eigenfunction 
and the discrete random eigenfunction without the dummy variable.

\begin{lemma}
\label{lemma-5.6}
Suppose $\lambda_D^{\ssup k}$ is simple. Then
\begin{equation}
\label{E:5.28aa}
\lim_{\gamma\downarrow 0}\,\limsup_{\epsilon\downarrow 0}
\sum_{m=1}^{|D_\epsilon|}\int_{B_\epsilon(x_m)}\textd y\,\,
\E\biggl(\,\bigl|\xi(x_m)-U(\epsilon x_m)\bigr|^2
\Bigr|\varphi_D^{\ssup{k}}(y)^2
-\epsilon^{-d}g_{D_\epsilon,\xi}^{\ssup {k}}({x_m})^2
\1_{A_{k,\epsilon,\gamma}}\Bigr|^2\biggr)=0.
\end{equation}
\end{lemma}

\begin{proofsect}{Proof}
Recall the setting of Corollary~\ref{cor3.8} and, in particular, given (a choice of) the scaled discrete eigenfunctions
$\epsilon^{-d/2}g_{D_\epsilon,\xi}^{\ssup 1},
\dots,\epsilon^{-d/2}g_{D_\epsilon,\xi}^{\ssup k}$, let $\wt
g_{1,\xi}^{\,\epsilon},\dots,\wt g_{k,\xi}^{\,\epsilon}$ denote their
continuum interpolations. 
Then \eqref{E:3.14uu} gives
\begin{equation}
\label{E:5.32}
\sum_{m=1}^{|D_\epsilon|}\int_{B_\epsilon(x_m)}\textd y\,\,
\E\biggl(\,\Bigr|\,\wt g_{k,\xi}^{\,\epsilon}(y)
-\epsilon^{-d/2}g_{D_\epsilon,\xi}^{\ssup {k}}(x_m)\Bigr|^2
\1_{A_{k,\epsilon,\gamma}}\biggr)\le C(d)\E\Bigl(\Vert\nabla^{\ssup{\textd}}
g_{D_\epsilon,\xi}^{\ssup {k}}\Vert_2^2\1_{A_{k,\epsilon,\gamma}}\Bigr),
\end{equation}
which tends to zero proportionally to~$\epsilon^2$, 
due to Corollary~\ref{lemma-regularity}. 
Thus it suffices to show that the following
tends to zero as $\epsilon\downarrow 0$ and $\gamma\downarrow 0$:
\begin{equation}
\label{E:5.31}
\begin{aligned}
\int_{D}\textd y\,\,
\E\biggl(\,\bigl|\xi(x_m)-U(\epsilon x_m)\bigr|^2
\Bigr|\,&\varphi_D^{\ssup{k}}(y)^2
-\wt g_{k,\xi}^{\,\epsilon}(y)^2\1_{A_{k,\epsilon,\gamma}}
\Bigr|^2\biggr)\\
&\quad\le 
\epsilon^{-2\kappa}
\BbbP(A_{k,\epsilon,\gamma}^\cc)
\|\varphi_D^{\ssup{k}}(y)\|_{L^4}^4\\
&\qquad+
\E\bigl(\|\xi-U(\epsilon \cdot)\bigr\|_{\epsilon,r}^r
\1_{A_{k,\epsilon,\gamma}}\bigr)^{2/r}
\E\Bigl(\,\bigl\|\,\bigl|\varphi_D^{\ssup{k}}\bigr|
-\bigl|\wt g_{k,\xi}^{\,\epsilon}\bigr|
\bigr\|_{L^r}^r \1_{A_{k,\epsilon,\gamma}}\Bigr)^{2/r}\\
&\qquad\qquad\times
\E\Bigl(\,\bigr\|\,\bigl|\varphi_D^{\ssup{k}}\bigr|
+\bigl|\wt g_{k,\xi}^{\,\epsilon}\bigr|\bigr\|_{L^{2r'}}^{2r'} 
\1_{A_{k,\epsilon,\gamma}}\Bigr)^{1/r'}.
\end{aligned}
\end{equation}
The first term on the right-hand side tends to zero as
 $\epsilon\downarrow 0$ because of Lemma~\ref{unif-g} and the 
boundedness of $\varphi_D^{\ssup{k}}$.
As for the second term, the definition of $A_{k,\epsilon,\gamma}$ and 
Proposition~\ref{prop2} imply
that the all the random variables in the expectations are bounded. 
As~$\lambda_D^{\ssup k}$ is simple,
Corollary~\ref{cor3.8} guarantees that when $\xi\in A_{k,\epsilon,\gamma}$ and $\gamma$ and $\epsilon$ are small, 
$\{\wt g_{j,\xi}^{\,\epsilon}\}_{j=1}^\ell$ projects almost
entirely onto the closed linear span of
$\{\varphi_D^{\ssup j}\}_{j=1}^\ell$ for 
both~$\ell=k-1$ and~$\ell=k$. 
This implies that we can make
\begin{equation}
\bigl\Vert \,
|\wt g_{k,\xi}^{\,\epsilon}|-|\varphi_D^{\ssup k}|\, 
\Vert_{L^2(D)}1_{A_{k,\epsilon,\gamma}}
\end{equation}
as small as we wish by making $\gamma$ and $\epsilon$ small. 
Since the H\"older inequality yields
\begin{equation}
 \bigl\Vert \,
|\wt g_{k,\xi}^{\,\epsilon}|-|\varphi_D^{\ssup k}|\,\bigr
\Vert_{L^r}^r\le \bigl\Vert \,
|\wt g_{k,\xi}^{\,\epsilon}|-|\varphi_D^{\ssup k}|\,\bigr
\Vert_{L^2}^{1/2}\bigl\Vert \,
|\wt g_{k,\xi}^{\,\epsilon}|-|\varphi_D^{\ssup k}|\,\bigr
\Vert_{L^{2(r-1)}}^{1/2}
\end{equation}
and $L^{2(r-1)}$-norm above is bounded due to Lemma~\ref{prop2}, 
we are done. 
\end{proofsect}

\begin{proofsect}{Proof of Proposition~\ref{prop-5.4}}
Combining Lemmas~\ref{lemma-5.5} and~\ref{lemma-5.6}, and using
that the conditional expectation is a contraction in~$L^2(\BbbP)$, 
we get
\begin{equation}
\label{E:5.33}
\sum_{m=1}^{|D_\epsilon|}\int_{B_\epsilon(x_m)}\textd y\,\,
\,\E
\biggl(\,\Bigl| \, \epsilon^{-d}Z_m^{\ssup k}-
\bigl(\xi(x_m)-U(\epsilon x_m)\bigr)
\varphi_D^{\ssup{k}}(y)^2\Bigr|^2\biggr)\,
\underset{\epsilon\downarrow0}\longrightarrow\,0
\end{equation}
for both $k={i,j}$. The claim now reduces to
\begin{equation}
\sum_{m=1}^{|D_\epsilon|}\int_{B_\epsilon(x_m)}\textd y\,
\bigl|V(y)-V(\epsilon
x_m)\bigr|\varphi_D^{{\ssup{i}}}(y)^2
\varphi_D^{{\ssup{j}}}(y)^2\,
\underset{\epsilon\downarrow0}\longrightarrow\,0,
\end{equation}
which follows by uniform continuity of $y\mapsto V(y)$ and the 
boundedness of the eigenfunctions.
\end{proofsect}

\begin{proofsect}{Proof of Theorem~\ref{thm1.2}}
The condition~(2) of the Martingale Central Limit Theorem is 
verified in Proposition~\ref{MartCLT2}. 
Thanks to Proposition~\ref{prop-5.4} and the fact that 
$|B_\epsilon({x_m})|=\epsilon^d$,
\begin{equation}
\epsilon^{-d}\sum_{m=1}^{|D_\epsilon|}\E\bigl(Z_m^{{\ssup{k_i}}}
Z_m^{{\ssup{k_j}}}\big|\FF_{m-1}\bigr)
\,\underset{\epsilon\downarrow0}\longrightarrow\,
\int_{D}V(y)\varphi_D^{\ssup{k_i}}(y)^2\varphi_D^{\ssup{k_j}}(y)^2\,
\textd y
\end{equation}
in $L^1(\BbbP)$ and thus in probability. This verifies the 
condition~(1) of the Martingale Central Limit Theorem and 
the result follows.
\end{proofsect}

\appendix
\section{Appendix}
\noindent
Here we collect some proofs from earlier parts of this paper. We begin by the proof of the Sobolev inequality.

\begin{proofsect}{Proof of Lemma~\ref{lemma-Sobolev}}
Since~$D$ is bounded we may regard~$D_\epsilon$ as a subset of the 
torus $\T_\epsilon:=\Z^d/(L\Z)^d$, where~$L$ is an integer at most twice 
the $\ell^\infty$-diameter of~$D_\epsilon$. This makes the discrete Fourier 
transform conveniently available. Writing
\begin{equation}
\hat f(k):=|\T_\epsilon|^{-1/2}\sum_{x\in\T_\epsilon}
\texte^{2\pi\texti k\cdot x/L}f(x),\qquad k\in\T_\epsilon,
\end{equation}
we get 
$\Vert \hat f\Vert_{\T_\epsilon,2}=\Vert f\Vert_{\T_\epsilon,2}$ 
and $\Vert \hat f\Vert_{\T_\epsilon,\infty}
\le c(D)\epsilon^{-d/2}\Vert f\Vert_{\T_\epsilon,1}$.
The Riesz-Thorin Interpolation Theorem then shows
\begin{equation}
\label{E:3.12q}
\tilde c(D,q)\Vert \hat f\Vert_{\T_\epsilon,q}
\le (\epsilon^{-d/2})^{\frac{q-2}q}\Vert f\Vert_{\T_\epsilon,p},
\end{equation}
where~$\tilde c(D,q)>0$ and every $q\in[2,\infty]$ and~$p$ such 
that $\ffrac1p+\ffrac1q=1$. As $\hat{\hat{f}}(x)=f(-x)$,
we may freely interchange~$\hat f$ with~$f$ in~\eqref{E:3.12q}. 

Let $\hat a_\epsilon(k):=\epsilon^{-2}\sum_{j=1}^d2\sin(\pi k_j/L)^2$ 
be the eigenvalue of~$-\epsilon^{-2}\Deltad$ on~$\T_\epsilon$ 
associated with the $k$-th Fourier mode. Applying \eqref{E:3.12q} 
and the H\"older inequality, for any~$q\ge2$ we get
\begin{equation}
\begin{aligned}
\tilde c(D,q)(\epsilon^{d/2})^{\frac{q-2}q}
\Vert f\Vert_{\T_\epsilon,q}
&\le \Vert \hat f\Vert_{\T_\epsilon,\frac q{q-1}}
\\ 
&\le \Vert(1+\hat a_\epsilon)^{-1/2}
\Vert_{\T_\epsilon,\frac{2q-2}{q-2}}^{\frac{q-2}{2q}}\,
\Vert(1+\hat a_\epsilon)^{1/2}\hat f\Vert_{\T_\epsilon,2}\\
&=\Vert(1+\hat a_\epsilon)^{-1/2}
\Vert_{\T_\epsilon,\frac{2q-2}{q-2}}^{\frac{q-2}{2q}}
\bigl(\Vert f\Vert_{\T_\epsilon,2}^2+\epsilon^{-2}
\Vert\nabla^{\ssup{\textd}}f\Vert_{\T_\epsilon,2}^2\bigr)^{1/2}
\end{aligned}
\end{equation}
Comparing with \eqref{E:3.10a}, it thus suffices to show that
\begin{equation}
\sup_{0<\epsilon<1}\sum_{k\in\T_\epsilon}
\bigl(1+\hat a_\epsilon(k)\bigr)^{\frac{q-1}{q-2}}<\infty.
\end{equation}
As~$\epsilon L$ is bounded between two positive numbers, this is 
equivalent to summability of $|k|^{-2\frac{q-1}{q-2}}$ 
on~$k\in\Z^d\setminus\{0\}$. This requires $2\frac{q-1}{q-2}>d$ 
which in~$d\ge3$ needs $q<\frac{2d}{d-2}$.
\end{proofsect}


Our next item of business is optimality of the moment condition and the effect of the truncation.
Let us first check that our moment assumption is nearly optimal for 
Theorem~\ref{thm1.1}. For the cases $d=1$ and 2, it is only a little more
than the natural integrability assumption. 
Let $d\ge 3$ and suppose that the distributions of $\xi^{\ssup \epsilon}(x)$ 
(${x\in D_\epsilon}$) 
depend neither on $x\in D_\epsilon$ nor on $\epsilon>0$. 
If we assume 
$\E[\xi(x)_-^K]=\infty$ for some $K<d/2$ in addition, then
\begin{equation}
\int_0^\infty t^{K-1} \BbbP(\xi_-(x) > t) \d t=\infty
\quad\Rightarrow \quad
\limsup_{t\to \infty}t^{-K'}\BbbP(\xi_-(x) > t)>0
\end{equation}
for any $K'>K$. Taking $K'< d/2$, we find  
\begin{equation}
\begin{split}
\limsup_{\epsilon\downarrow 0}
 \BbbP\Bigl(\min_{x\in D_\epsilon}\xi(x)\le -\epsilon^{-\kappa}\Bigr)
&=1-\liminf_{\epsilon\downarrow 0}
 \prod_{x\in D_\epsilon}\BbbP\bigl(\xi_-(x)\le \epsilon^{-\kappa}\bigr)\\
&\ge 1-\liminf_{\epsilon\downarrow 0}
\Bigl(1-\epsilon^{\kappa K'}\Bigr)^{{|D_\epsilon|}}>0 
\end{split}
\label{min-tail}
\end{equation}
for $2< \kappa <d/K'$. 
Suppose $\xi(x)\le -\epsilon^{-\kappa}$ at 
$x\in D_\epsilon$. 
Then, by simply taking $h_1=1_{\{x\}}$ in~\eqref{E:3.22} 
with $k=1$, we obtain  
\begin{equation}
\lambda^{\ssup 1}_{D_\epsilon,\xi}
\le \epsilon^{-2}\|\nabla^{\ssup{\textd}}1_{\{x\}}\|_2^2
-\langle 1_{\{x\}}, \xi 1_{\{x\}}\rangle\le -\epsilon^{-\kappa}/2.
\end{equation}
This and~\eqref{min-tail} implies that Theorem~\ref{thm1.1} fails to
hold.

Next, we shall show that the truncation may affect the mean value
$\E[\lambda^{\ssup 1}_{D_\epsilon,\xi}]$. 
Suppose for simplicity that $\{\xi(x)\}_{x\in \Z^d}$ are identically
distributed and 
\begin{equation}
\BbbP(\xi(x) \le -r)=|r|^{-K}\wedge 1
\end{equation}
for some $K>1\vee d/2$. This distribution clearly satisfies
Assumption~\ref{ass1} with $U$ being a constant function and
\begin{equation}
 \BbbP\left(\min_{x\in D_\epsilon}\xi(x) \le -r\right)
 =(1-|r|^{-K}\wedge 1)^{\#D_\epsilon}\ge c\epsilon^d |r|^{-K}
\end{equation}
provided that the last line is much smaller than 1. 
As is seen in the above
argument, if $\xi(x)\le -M\epsilon^{-2}$ 
for some large $M>0$ and $x\in D_\epsilon$, then $h_1=1_{\{x\}}$ is almost
the optimal choice in~\eqref{E:3.22} and 
$\lambda^{\ssup 1}_{D_\epsilon,\xi}\le -\xi(x)/2$. 
\begin{equation}
\begin{split}
\E\bigl[\lambda^{\ssup 1}_{D_\epsilon,\xi}\bigr]
&\le \frac{1}{2}\E\left[\min_{x\in D_\epsilon}\xi(x)\colon
\min_{x\in D_\epsilon}\xi(x)\le -M\epsilon^{-2}\right]\\
&=- \int_{-\infty}^{-M\epsilon^{-2}}
\BbbP\left(\min_{x\in D_\epsilon}\xi(x)\le -r \right) \d r\\ 
&\lesssim -\epsilon^{-d}\int_{M\epsilon^{-2}}^{\infty}
r^{-K} \d r\asymp -\epsilon^{-d+2(K-1)}.
\end{split}
\end{equation}
If $K<d/2+1$ (this is possible when $d\ge 3$), 
the right-hand side goes to $-\infty$.


\end{document}